%% file: SlepianEigenvalueBounds-main.tex
\pdfoutput=1
\documentclass[10pt]{article}

\usepackage{palatino}
\usepackage{float}

\input{macro}

\newcommand\blfootnote[1]{
  \begingroup
  \renewcommand\thefootnote{}\footnote{#1}
  \addtocounter{footnote}{-1}
  \endgroup
}

\begin{document}
\title{Improved bounds for the eigenvalues of prolate spheroidal wave functions and discrete prolate spheroidal sequences}

\vspace{2mm}
\author{Santhosh Karnik, Justin Romberg, Mark A. Davenport}

\maketitle

\begin{abstract}
The discrete prolate spheroidal sequences (DPSSs) are a set of orthonormal sequences in $\ell_2(\Z)$ which are strictly bandlimited to a frequency band $[-W,W]$ and maximally concentrated in a time interval $\{0,\ldots,N-1\}$. The timelimited DPSSs (sometimes referred to as the Slepian basis) are an orthonormal set of vectors in $\mathbb{C}^N$ whose discrete time Fourier transform (DTFT) is maximally concentrated in a frequency band $[-W,W]$. Due to these properties, DPSSs have a wide variety of signal processing applications. The DPSSs are the eigensequences of a timelimit-then-bandlimit operator and the Slepian basis vectors are the eigenvectors of the so-called prolate matrix. The eigenvalues in both cases are the same, and they exhibit a particular clustering behavior -- slightly fewer than $2NW$ eigenvalues are very close to $1$, slightly fewer than $N-2NW$ eigenvalues are very close to $0$, and very few eigenvalues are not near $1$ or $0$. This eigenvalue behavior is critical in many of the applications in which DPSSs are used. There are many asymptotic characterizations of the number of eigenvalues not near $0$ or $1$. In contrast, there are very few non-asymptotic results, and these don't fully characterize the clustering behavior of the DPSS eigenvalues. In this work, we establish two novel non-asymptotic bounds on the number of DPSS eigenvalues between $\epsilon$ and $1-\epsilon$. Also, we obtain bounds detailing how close the first $\approx 2NW$ eigenvalues are to $1$ and how close the last $\approx N-2NW$ eigenvalues are to $0$. Furthermore, we extend these results to the eigenvalues of the prolate spheroidal wave functions (PSWFs), which are the continuous-time version of the DPSSs. Finally, we present numerical experiments demonstrating the quality of our non-asymptotic bounds on the number of DPSS eigenvalues between $\epsilon$ and $1-\epsilon$.
\end{abstract}

\blfootnote{S. Karnik, J. Romberg, and M. A. Davenport are with the School of Electrical and Computer Engineering, Georgia Institute of Technology, Atlanta, GA, 30332 USA (e-mail: skarnik1337@gatech.edu, jrom@ece.gatech.edu, mdav@gatech.edu). This work was supported by a grant from Lockheed Martin and a gift from the Alfred P. Sloan Foundation.}

\input{Intro}

\input{MathematicalPreliminaries}

\input{EigenvalueBehavior}

\input{MainProof}

\input{DPSSEigenvalueBounds}

\input{PSWFEigenvalueBounds}

\input{NumericalResults}

\frenchspacing
\bibliographystyle{unsrt}
\bibliography{bibfileSlepianEigenvalueBounds}

\appendix
\input{Zolotarev}

\input{LowRankDisplacement}

\input{SincApprox}

\end{document}

%% file: macro.tex
\usepackage{url}
\usepackage{amsmath,amsthm,amssymb,amsbsy}
\usepackage{mathdots}
\usepackage{paralist}
\usepackage{xcolor}
\usepackage{color}
\usepackage{graphicx}
\usepackage{algorithm,algpseudocode}
\usepackage{comment}

\usepackage{fancyhdr}
\usepackage{cite}
\usepackage{cleveref}
\usepackage{enumerate}

\newtheorem{lemma}{Lemma}
\newtheorem{cor}{Corollary}

\newtheorem{theorem}{Theorem}

\theoremstyle{remark}

\newcommand{\R}{\mathbb{R}}

\newcommand{\C}{\mathbb{C}}

\newcommand{\Z}{\mathbb{Z}}
\newcommand{\N}{\mathbb{N}}

\newcommand{\vct}[1]{\boldsymbol{#1}}
\newcommand{\mtx}[1]{\boldsymbol{#1}}

\newcommand{\inner}[1]{\left<#1\right>}

\newcommand{\trace}{\operatorname{trace}}

\newcommand{\rank}{\operatorname{rank}}

\newcommand{\Spec}{\operatorname{Spec}}
\newcommand{\set}[1]{\mathcal{#1}}

\DeclareMathOperator*{\maximize}{\text{maximize}}

\newcommand{\floor}[1]{\left\lfloor #1 \right\rfloor}
\newcommand{\ceil}[1]{\left\lceil #1 \right\rceil}
\newcommand{\eps}{\epsilon}
\renewcommand{\j}{\mathbf{i}}

\newcommand{\calB}{\mathcal{B}}

\newcommand{\calI}{\mathcal{I}}

\newcommand{\calT}{\mathcal{T}}

\newcommand{\calBc}{\mathcal{B}}
\newcommand{\calTc}{\mathcal{T}}

\newcommand{\vs}{\vct{s}}

\newcommand{\mA}{\mtx{A}}
\newcommand{\mB}{\mtx{B}}
\newcommand{\mC}{\mtx{C}}
\newcommand{\mD}{\mtx{D}}

\newcommand{\mP}{\mtx{P}}

\newcommand{\mU}{\mtx{U}}
\newcommand{\mV}{\mtx{V}}
\newcommand{\mW}{\mtx{W}}
\newcommand{\mX}{\mtx{X}}
\newcommand{\mY}{\mtx{Y}}

\newcommand{\mId}{{\bf I}}

\newcommand{\setI}{\set{I}}

\graphicspath{{./figs/}}

\newlength{\imgwidth}
\newboolean{twoColVersion}
\setboolean{twoColVersion}{false}
\newcommand{\twoCol}[2]{\ifthenelse{\boolean{twoColVersion}} {#1} {#2} }

%% file: Intro.tex
\section{Introduction}
\label{sec:Intro}
A fundamental fact of Fourier analysis is that no non-zero signal can be simultaneously bandlimited and timelimited. Thus, a compactly supported non-zero function cannot have a compactly supported Fourier transform, and a non-zero function whose Fourier transform is compactly supported cannot itself be compactly supported. Between 1960 and 1978, Landau, Pollak, and Slepian published a series of seminal papers \cite{SlepianPollakI, LandauPollakII, LandauPollakIII, SlepianIV, SlepianV} exploring to what extent a bandlimited signal can be timelimited and to what extent a timelimited signal can be bandlimited. They formulate both of these questions as eigenproblems involving bandlimiting and timelimiting operators. In the continuous-time case, the eigenfunctions which are bandlimited to the frequency band $[-\Omega,\Omega]$ and are maximally concentrated in the time interval $[-\tfrac{T}{2},\tfrac{T}{2}]$ are known as the prolate spheroidal wave functions (PSWFs). In the discrete-time case, the eigensequences which are bandlimited to frequencies $[-W,W]$ and are maximally concentrated in the time indices $\{0,1,\ldots,N-1\}$ are known as the discrete prolate spheroidal sequences (DPSSs).

By truncating the DPSSs, one can form the Slepian basis vectors, which are an efficient basis for representing a window of samples from bandlimited signals \cite{DavenportWakin12,ZhuWakin17,KarnikFST}. As such, Slepian basis vectors can be used in a variety of applications. Some classic applications include prediction of bandlimited signals based on past samples \cite{SlepianV} and Thomson's multitaper method for spectral analysis \cite{Thomson82}. More recent applications include time-variant channel estimation \cite{Zemen05,Zemen07}, wideband
compressive radio receivers \cite{Davenport10}, compressed sensing of analog signals \cite{DavenportWakin12}, target detection \cite{Yin12,ZhuWakin15}, and a fast method \cite{Matthysen16} for computing Fourier extension series coefficients \cite{Huybrechs10, Adcock14}. 

The eigenvalues associated with the PSWFs as well as the DPSSs (which we will refer to in this paper as the PSWF eigenvalues and the DPSS eigenvalues respectively) exhibit a particular clustering behavior. All of these eigenvalues are strictly between $0$ and $1$. Furthermore, most of these eigenvalues are either very close to $1$ or very close to $0$, and very few of these eigenvalues are not near $1$ or $0$. This clustering behavior plays a critical role in many applications. In \cite{Abreu17}, it is shown that the bias of Thomson's multitaper spectral estimator depends on the sum of the leading DPSS eigenvalues. In \cite{DavenportWakin12}, it is shown that the number of DPSS eigenvalues not near $0$ determines the effective dimension of a vector of samples from a bandlimited signal. Also, the sum of the trailing DPSS eigenvalues bounds the error in approximating a vector of samples from a bandlimited signal by a linear combination of the leading DPSSs. In \cite{Matthysen16,KarnikFST,KarnikSampTA19}, the fact that only a small number of DPSS eigenvalues are not near $1$ or $0$ is exploited to perform fast computations. As such, the behavior of the DPSS eigenvalues is of considerable interest. There are many results (both asymptotic and non-asymptotic) regarding the eigenvalues associated with the PSWFs \cite{SlepianPollakI,LandauPollakIII,LandauWidom80,Hogan11,Osipov13,Israel15,Bonami17,Bonami18}. However, there are far fewer results regarding the DPSS eigenvalues. Furthermore, the existing non-asymptotic results don't fully capture the behavior of the DPSS eigenvalues or the PSWF eigenvalues. 

The main contribution of this work is establishing a novel non-asymptotic bound on the number of DPSS eigenvalues which are not close to $1$ or $0$, as well as non-asymptotic bounds on the DPSS eigenvalues themselves. We show that for any threshold $\eps \in (0,\tfrac{1}{2})$, the number of DPSS eigenvalues which lie in the interval $(\eps,1-\eps)$ is bounded above by a quantity that scales like $O(\log(NW)\log\tfrac{1}{\eps})$. This improved bound is similar to the known asymptotic result and is a significant improvement over existing non-asymptotic bounds which scale like $O(\log(NW)/\eps)$ or $O(\log N \log \tfrac{1}{\eps})$. Furthermore, the fact that our bound depends on the time-bandwidth product $NW$ as opposed to the time duration $N$ allows us to obtain non-asymptotic bounds on the PSWF eigenvalues, which are also similar to the known asymptotic result and a substantial improvement over existing non-asymptotic results.

The rest of this work is organized as follows. In Section~\ref{sec:Preliminaries}, we define DPSSs, Slepian basis vectors, and PSWFs, as well as review their basic properties. In Sections~\ref{sec:PrevDPSSBounds} and \ref{sec:PrevPSWFBounds}, we will outline existing results on the DPSS eigenvalues and the PSWF eigenvalues respectively. In Section~\ref{sec:NewResults}, we state our new results. In Section~\ref{sec:MainProof}, we prove two novel non-asymptotic bounds on the number of DPSS eigenvalues that are not within $\eps$ of $1$ or $0$. In Section~\ref{sec:DPSSEigenvalueBounds}, we use the two new bounds to derive non-asymptotic bounds on the DPSS eigenvalues themselves, as well as the sum of the leading and trailing DPSS eigenvalues. In Section~\ref{sec:PSWFProofs}, we apply the results on DPSS eigenvalues to prove our results on PSWF eigenvalues. Finally, we conclude the paper in Section~\ref{sec:NumericalResults} with some numerical results to demonstrate the quality of our non-asymptotic bounds on the number of DPSS eigenvalues in $(\eps,1-\eps)$.

%% file: MathematicalPreliminaries.tex
\section{Mathematical Preliminaries}
\label{sec:Preliminaries}
\subsection{Discrete prolate spheroidal sequences and Slepian basis vectors}
\label{sec:DPSS}
We begin by defining the discrete prolate spheroidal sequences (DPSSs) and Slepian basis vectors and reviewing some of their key properties. For a discrete signal $x \in \ell_2(\Z)$, we define its discrete time Fourier transform (DTFT) $\widehat{x} \in L_2([-\tfrac{1}{2},\tfrac{1}{2}])$ by $$\widehat{x}(f) = \sum_{n = -\infty}^{\infty}x[n]e^{-\j 2\pi fn} \quad \text{for} \quad f \in [-\tfrac{1}{2},\tfrac{1}{2}],$$ where we use the notation $\j = \sqrt{-1}$. The inverse DTFT is given by $$x[n] = \int_{-1/2}^{1/2}\widehat{x}(f)e^{\j 2\pi fn}\,df \quad \text{for} \quad n \in \Z.$$ With these definitions, any $x,x' \in \ell_2(\Z)$ satisfy the Parseval-Plancherel identity $\inner{x,x'}_{\ell_2(\Z)} = \inner{\widehat{x},\widehat{x}'}_{L_2([-1/2,1/2])}$. For any $N \in \N$, we say that $x \in \ell_2(\Z)$ is timelimited to $n \in \{0,\ldots,N-1\}$ if $x[n] = 0$ for $n \in \Z \setminus \{0,\ldots,N-1\}$. Also, for any $W \in (0,\tfrac{1}{2})$, we say that $x \in \ell_2(\Z)$ is bandlimited to $|f| \le W$ if $\widehat{x}(f) = 0$ for $|f| > W$. 

We can now ask the question ``what signal bandlimited to $|f| \le W$ has a maximum concentration of energy over the time indices $n \in \{0,\ldots,N-1\}$?'', i.e., $$\maximize_{x \in \ell_2(\Z)} \sum_{n = 0}^{N-1}|x[n]|^2 \quad \text{subject to} \quad \|x\|_{\ell_2(\Z)}^2 = 1 \quad \text{and} \quad \widehat{x}(f) = 0 \ \text{for} \ |f| > W.$$ 

To help answer this question, we define two self-adjoint operators. For a given $N \in \N$ we define a timelimiting operator $\calT_N : \ell_2(\Z) \to \ell_2(\Z)$ by $$(\calT_N x)[n] = \begin{cases}x[n] & \text{if} \ n \in \{0,\ldots,N-1\} \\ 0 & \text{if} \ n \in \Z\setminus\{0,\ldots,N-1\} \end{cases}.$$ For a given bandwidth parameter $W \in (0,\tfrac{1}{2})$, we define a bandlimiting operator $\calB_W : \ell_2(\Z) \to \ell_2(\Z)$ by $$(\calB_W x)[n] = \sum_{m = -\infty}^{\infty}\dfrac{\sin[2\pi W(m-n)]}{\pi(m-n)}x[m] \quad \text{for} \quad n \in \Z.$$ Note that the DTFT of $\calB_W x$ satisfies $\widehat{\calB_W x}(f) = \widehat{x}(f)$ for $|f| \le W$ and $\widehat{\calB_W x}(f) = 0$ for $|f| > W$. 

For bandlimited signals $x \in \ell_2(\Z)$, we can write $$\sum_{n = 0}^{N-1}|x[n]|^2 = \inner{x,\calT_Nx}_{\ell_2(\Z)} = \inner{\calB_Wx,\calT_N\calB_Wx}_{\ell_2(\Z)} = \inner{x,\calB_W\calT_N\calB_Wx}_{\ell_2(\Z)}.$$ Subject to the constraint $\|x\|_{\ell_2(\Z)}^2 = 1$, this is maximized by the eigensequence of $\calB_W\calT_N\calB_W$ corresponding to the largest eigenvalue. Slepian defined the discrete prolate spheroidal sequences (DPSSs) $s_0,\ldots,s_{N-1} \in \ell_2(\Z)$ as the $N$ orthonormal eigensequences of $\calB_W\calT_N\calB_W$ corresponding to non-zero eigenvalues. The corresponding eigenvalues $1 > \lambda_0 > \lambda_1 > \cdots > \lambda_{N-1} > 0$ are referred to as the DPSS eigenvalues and are sorted in descending order. Note that the notation $s_k$ and $\lambda_k$ hides the dependence on $N$ and $W$. When it is necessary to make this dependence explicit, we will use the expanded notation $s_k(N,W)$ and $\lambda_k(N,W)$ respectively. 

Slepian \cite{SlepianV} showed that these eigenvalues are all distinct and strictly between $0$ and $1$. In addition to $s_0$ being the bandlimited sequence in $\ell_2(\Z)$ with a maximal concentration of energy in $\{0,\ldots,N-1\}$, it is also true that for each $k = 1,\ldots,N-1$, $s_k$ is the bandlimited sequence in $\ell_2(\Z)$ with a maximal concentration of energy in $\{0,\ldots,N-1\}$ subject to the additional constraint of being orthogonal to $s_0,\ldots,s_{k-1}$. Furthermore $\lambda_k$ is equal to the amount of energy $s_k$ has in the time interval $\{0,\ldots,N-1\}$.

We can also ask the question ``what signal timelimited to $n \in \{0,\ldots,N-1\}$ has a maximum concentration of energy in the frequency band $|f| \le W$?'', i.e. $$\maximize_{x \in \ell_2(\Z)}\int_{-W}^{W}\left|\widehat{x}(f)\right|^2\,df \quad \text{subject to} \quad \|x\|_{\ell_2(\Z)}^2 = 1 \quad \text{and} \quad x[n] = 0 \ \text{for} \ n \in \Z\setminus\{0,\ldots,N-1\}.$$

For timelimited signals $x \in \ell_2(\Z)$, we can write $$\int_{-W}^{W}\left|\widehat{x}(f)\right|^2\,df = \inner{\widehat{x},\widehat{\calB_Wx}}_{L_2([-1/2,1/2])} = \inner{x,\calB_Wx}_{\ell_2(\Z)} = \inner{\calT_Nx,\calB_W\calT_Nx}_{\ell_2(\Z)} = \inner{x,\calT_N\calB_W\calT_Nx}_{\ell_2(\Z)}.$$ Since $\calT_N\calB_W\calT_N$ is self-adjoint, the sequence $x \in \ell_2(\Z)$ which solves the above maximization problem is the eigensequence of the operator $\calT_N\calB_W\calT_N$ corresponding to the largest eigenvalue.

Clearly, the range of $\calT_N\calB_W\calT_N$ and the orthogonal complement of the kernel of $\calT_N\calB_W\calT_N$ is the $N$-dimesional space of timelimited signals. Hence, we can reduce this eigenproblem on $\ell_2(\Z)$ to an eigenproblem on $\R^N$. With respect to the Euclidean basis for the space of timelimited signals, the matrix representation of $\calT_N\calB_W\calT_N$ is given by 
\begin{equation}
\mB[m,n] = \dfrac{\sin[2\pi W(m-n)]}{\pi(m-n)} \quad \text{for} \quad m,n = 0,\ldots,N-1. 
\label{eq:Prolate}
\end{equation} 
This matrix $\mB \in \R^{N \times N}$ is known in the literature as the prolate matrix \cite{Varah93,Bojanczyk95}. The Slepian basis vectors $\vs_0,\ldots,\vs_{N-1} \in \R^N$ are the orthonormal eigenvectors of $\mB$, where again the eigenvalues $1 > \lambda_0 > \lambda_1 > \cdots > \lambda_{N-1} > 0$ are sorted in descending order. Note that the eigenvalues of $\mB$ are the same as the eigenvalues of $\calT_N\calB_W\calT_N$, which are the same as the eigenvalues of $\calB_W\calT_N\calB_W$. Hence, we can reuse the notation $\lambda_k$ for $k = 0,\ldots,N-1$ to denote the eigenvalues of $\mB$. The eigensequences $s'_0,\ldots,s'_{N-1} \in \ell_2(\Z)$ of $\calT_N\calB_W\calT_N$ are then given by $s'_k[n] = \vs_k[n]$ for $n \in \{0,\ldots,N-1\}$ and $s'_k[n] = 0$ for $n \in \Z\setminus\{0,\ldots,N-1\}$. Note that in addition to $s'_0$ being the timelimited sequence in $\ell_2(\Z)$ whose DTFT has a maximal concentration of energy in $[-W,W]$, it is also true that for each $k = 1,\ldots,N-1$, $s'_k$ is a timelimited sequence in $\ell_2(\Z)$  whose DTFT has a maximal concentration of energy in $[-W,W]$ subject to the additional constraint of being orthogonal to $s'_0,\ldots,s'_{k-1}$. Furthermore, the eigenvalue $\lambda_k$ is equal to the amount of energy that the DTFT of $s'_k$ has in the frequency band $[-W,W]$.

\subsection{Prolate spheroidal wave functions}
\label{sec:PSWF}
The concentration problems in the previous subsection can also be posed in the continuous case. For continuous signals in $L_2(\R)$, we define the continuous time Fourier transform and inverse Fourier transform by $$\widehat{y}(\omega) = \int_{-\infty}^{\infty}y(t)e^{-\j \omega t}\,dt \quad \text{for} \quad \omega \in \R \quad \text{and} \quad y(t) = \dfrac{1}{2\pi}\int_{-\infty}^{\infty}\widehat{y}(\omega)e^{\j \omega t}\,d\omega \quad \text{for} \quad t \in \R.$$ With these definitions, any $y,y' \in L_2(\R)$ satisfy the Parseval-Plancherel identity $\inner{y,y'}_{L_2(\R)} = \tfrac{1}{2\pi}\inner{\widehat{y},\widehat{y}'}_{L_2(\R)}$. 

Again, for a given bandlimit $\Omega > 0$ and duration $T > 0$, we can ask ``what signal bandlimited to $|\omega| \le \Omega$ has a maximum concentration of energy in the time interval $t \in [-\tfrac{T}{2},\tfrac{T}{2}]$?'', i.e. $$\maximize_{y \in L_2(\R)}\int_{-T/2}^{T/2}\left|y(t)\right|^2\,dt \quad \text{subject to} \quad \|y\|_{L_2(\R)}^2 = 1 \quad \text{and} \quad \widehat{y}(\omega) = 0 \ \text{for} \ |\omega| > \Omega.$$ 

Just as was done with the discrete case, we can define a self-adjoint timelimiting operator $\calTc_T : L_2(\R) \to L_2(\R)$ and a self-adjoint bandlimiting operator $\calBc_{\Omega} : L_2(\R) \to L_2(\R)$ by $$(\calTc_Ty)(t) = \begin{cases} y(t) & \text{if} \ |t| \le \tfrac{T}{2} \\ 0 & \text{if} \ |t| > \tfrac{T}{2}\end{cases} \quad \text{and} \quad (\calBc_{\Omega}y)(t) = \int_{-\infty}^{\infty}\dfrac{\sin[\Omega(t-t')]}{\pi(t-t')}y(t')\,dt' \quad \text{for} \quad t \in \R,$$ and write $$\int_{-T/2}^{T/2}\left|y(t)\right|^2\,dt = \inner{y,\calTc_{T}y}_{L_2(\R)} = \inner{\calBc_{\Omega}y,\calTc_{T}\calBc_{\Omega}y}_{L_2(\R)} = \inner{y,\calBc_{\Omega}\calTc_T\calBc_{\Omega}y}_{L_2(\R)}$$ for any bandlimited signal $y \in L_2(\R)$. 

The result is that the solution to this concentration problem is top eigenfunction of the operator $\calBc_{\Omega}\calTc_T\calBc_{\Omega}$. The orthonormal eigenfunctions $\psi_0, \psi_1, \ldots \in L_2(\R)$ of the self-adjoint operator $\calBc_{\Omega}\calTc_T\calBc_{\Omega}$ are known as the prolate spheroidal wave functions (PSWFs), and the corresponding eigenvalues $1 > \widetilde{\lambda}_0 > \widetilde{\lambda}_1 >  \cdots > 0$ are known as the PSWF eigenvalues, and are sorted in decreasing order.

It is easy to check that the timelimited PSWFs $\calTc_T\psi_k$ for $k \in \N$ are the eigenfunctions of the operator $\calTc_T\calBc_{\Omega}\calTc_T$ with the same corresponding eigenvalues $\widetilde{\lambda}_k$. The action of the operator $\calTc_T\calBc_{\Omega}\calTc_T$ on a signal $y \in L_2(\R)$ is given by $$(\calTc_T\calBc_{\Omega}\calTc_Ty)(t) = \int_{-T/2}^{T/2}\dfrac{\sin[\Omega(t-t')]}{\pi(t-t')}y(t')\,dt' \quad \text{for} \quad t \in [-\tfrac{T}{2},\tfrac{T}{2}].$$ With a simple change of variable, it can be shown that the eigenvalues of the above kernel integral operator only depend on the product $\Omega T$. When it is necessary to denote this dependence, we use the notation $\widetilde{\lambda}_k(c)$ to denote $\widetilde{\lambda}_k$ for values of $\Omega, T > 0$ which satisfy $\Omega T/2 = c$. 

As in the discrete case, $\psi_0$ is the signal bandlimited to $|\omega| \le \Omega$ whose energy is maximally concentrated in the time interval $t \in [-\tfrac{T}{2},\tfrac{T}{2}]$, and for each integer $k \ge 1$, $\psi_k$ is the signal bandlimited to $|\omega| \le \Omega$ whose energy is maximally concentrated in the time interval $t \in [-\tfrac{T}{2},\tfrac{T}{2}]$ subject to the additional constraint of being orthogonal to $\psi_0,\ldots,\psi_{k-1}$. Furthermore, the eigenvalue $\widetilde{\lambda}_k$ is the equal to the energy of $\psi_k$ in the time interval $t \in [-\tfrac{T}{2},\tfrac{T}{2}]$. Also, $\calTc_T\psi_0$ is the signal timelimited to $t \in [-\tfrac{T}{2},\tfrac{T}{2}]$ whose Fourier transform is maximally concentrated in the frequency band $|\omega| \le \Omega$, and for each integer $k \ge 1$, $\calTc_T\psi_k$ is the signal timelimited to $t \in [-\tfrac{T}{2},\tfrac{T}{2}]$ whose Fourier transform is maximally concentrated in the frequency band $|\omega| \le \Omega$ subject to the additional constraint of being orthogonal to $\calTc_T\psi_0,\ldots,\calTc_T\psi_{k-1}$. Furthermore, the eigenvalue $\widetilde{\lambda}_k$ is the equal to the fraction of energy that $\calTc_T\psi_k$ has in the frequency band $|\omega| \le \Omega$.

%% file: EigenvalueBehavior.tex
\section{Eigenvalue Behavior}
\label{sec:Results}

\subsection{Prior results on DPSS eigenvalues}
\label{sec:PrevDPSSBounds}
Showing that the DPSS eigenvalues are strictly between $0$ and $1$ is a trivial consequence of the facts that $\lambda_k = \left(\int_{-W}^{W}|\widehat{s}_k(f)|^2\,df\right)/\left(\int_{-1/2}^{1/2}|\widehat{s}_k(f)|^2\,df\right)$ and that $\widehat{\vs}_k(f)$ is a non-zero analytic function. It is also easy to check that the sum of all the DPSS eigenvalues is $\sum_{k = 0}^{N-1}\lambda_k = \trace(\mB) = 2NW$. What is perhaps more interesting is that the DPSS eigenvalues obey a particular clustering behavior. For any $\eps \in (0,\tfrac{1}{2})$, slightly fewer than $2NW$ eigenvalues lie in $[1-\eps,1)$, slightly fewer than $N-2NW$ eigenvalues lie in $(0,\eps]$, and very few eigenvalues lie in the transition region $(\eps,1-\eps)$. In Figure 1, we demonstrate this phenomenon by plotting the DPSS eigenvalues for $N = 1000$ and $W = \tfrac{1}{8}$ (so $2NW = 250$). The first $244$ eigenvalues lie in $[0.999,1)$ and the last $744$ eigenvalues lie in $(0,0.001]$. Only $12$ eigenvalues lie between $0.001$ and $0.999$. 

\begin{figure}
\label{fig:EigenvaluePlot}
\centering
\includegraphics[scale=0.35]{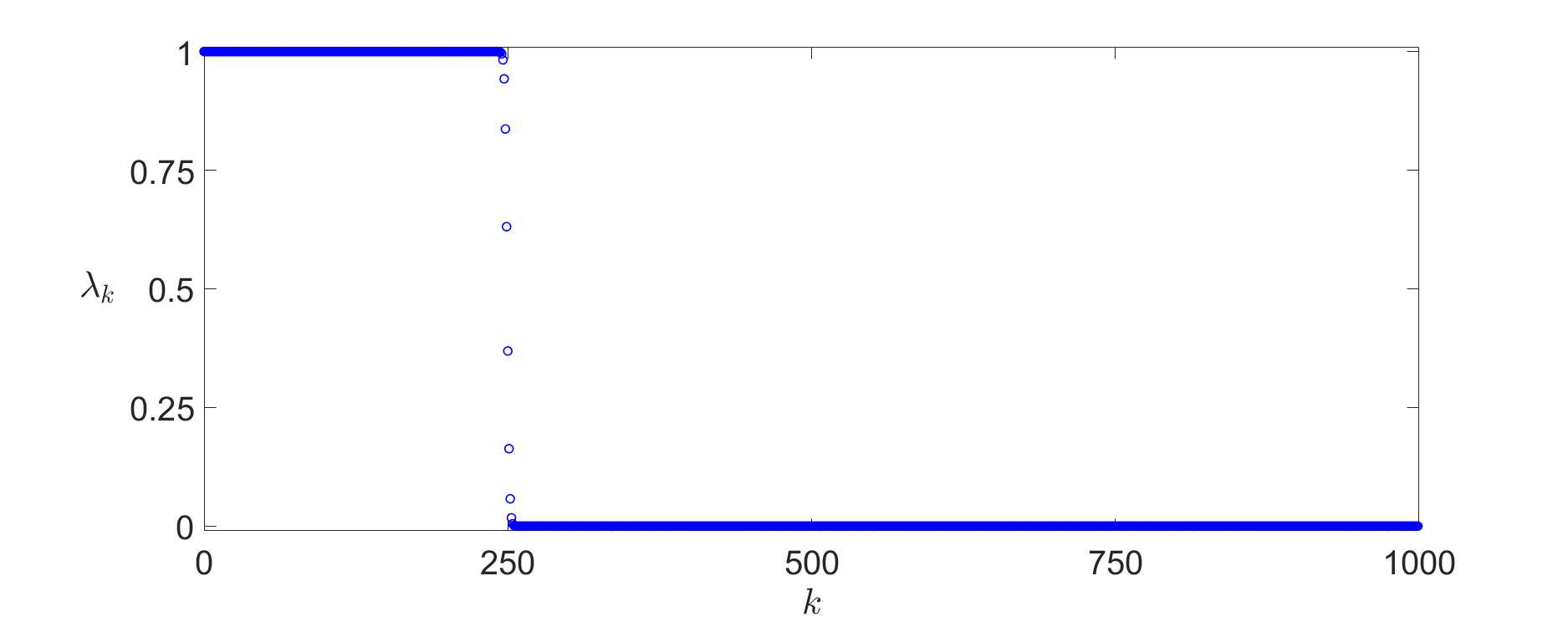}
\caption{A plot of the DPSS Eigenvalues $\{\lambda_k\}_{k = 0}^{N-1}$ for $N = 1000$ and $W = \tfrac{1}{8}$. These eigenvalues satisfy $\lambda_{243} \approx 0.9997$ and $\lambda_{256} \approx 0.0003$. Only $12$ of the $1000$ DPSS eigenvalues lie in $(0.001,0.999)$.}
\end{figure}

Experimentally, we can see that the width of this transition region behaves like $\#\{k : \eps < \lambda_k < 1-\eps\} = O(\log(NW) \log(\tfrac{1}{\eps}))$. This can be seen in Figures~\ref{fig:GapsizeVsN} and \ref{fig:GapsizeVsW} in Section~\ref{sec:NumericalResults}. Our main contribution will be to demonstrate this analytically, but before we do, we will briefly review some of the prior results along these lines.

We begin with the original results from Slepian \cite{SlepianV}. For any fixed $W \in (0,\tfrac{1}{2})$ and $b \in \R$, $$\lambda_{\floor{2NW+(b/\pi)\log N}} \sim \dfrac{1}{1+e^{b\pi}} \quad \text{as} \quad N \to \infty.$$ From this result, it is easy to show that for any fixed $W \in (0,\tfrac{1}{2})$ and $\eps \in (0,\tfrac{1}{2})$, $$\#\{k : \eps < \lambda_k < 1-\eps\} \sim \dfrac{2}{\pi^2}\log N \log\left(\dfrac{1}{\eps}-1\right) \quad \text{as} \quad N \to \infty.$$ This asymptotic bound on the width of the transition region correctly captures the logarithmic dependence on both $N$ and $\eps$, but not the dependence on $W$. Slepian also stated that if $0.2 < \lambda_k < 0.8$, then $$\lambda_k \approx \left[1+\exp\left(-\dfrac{\pi^2(2NW-k-\tfrac{1}{2})}{\log[8N\sin(2\pi W)]+\gamma}\right)\right]^{-1}$$ is a good approximation to $\lambda_k$ where $\gamma \approx 0.5772$ is the Euler-Mascheroni constant. This would suggest that $$\#\{k : \eps < \lambda_k < 1-\eps\} \approx \dfrac{2}{\pi^2}\log\left[8e^{\gamma}N\sin(2\pi W)\right]\log\left(\dfrac{1}{\eps}-1\right)$$ for $\eps \in (0.2,0.5)$. This correctly captures the logarithmic dependence on $N$, $W$, and $\eps$, but only holds for large values of $\eps$.

Very few papers provide non-asymptotic bounds regarding the width of the transition region $\#\{k : \eps < \lambda_k < 1-\eps\}$. Zhu and Wakin \cite{ZhuWakin15} showed that 
\begin{equation}
\#\{k : \eps \le \lambda_k \le 1-\eps\} \le \dfrac{\tfrac{2}{\pi^2}\log(N-1) + \tfrac{2}{\pi^2}\tfrac{2N-1}{N-1}}{\eps(1-\eps)}
\label{eq:ZhuWakin}
\end{equation}
for all integers $N \ge 2$, $W \in (0,\tfrac{1}{2})$, and $\eps \in (0,\tfrac{1}{2})$. This non-asymptotic bound correctly highlights the logarithmic dependence on $N$, but fails to capture the dependence on $W$. Also, the dependence on $\eps$ is $O(\tfrac{1}{\eps})$ as opposed to $O(\log \tfrac{1}{\eps})$. When $\eps$ is small, this bound is considerably worse than a $O(\log \tfrac{1}{\eps})$ bound. For example, when $N = 1000$, $W = \tfrac{1}{8}$, and $\eps = 10^{-3}$, this result becomes $\#\{k : \eps < \lambda_k < 1-\eps\} \le 1806$. More generally, when $\eps < \tfrac{2\log(N-1)}{\pi^2N}$, this bound is worse than the trivial bound of $\#\{k : \eps < \lambda_k < 1-\eps\} \le N$. 

Recently, Boulsane, Bourguiba, and Karoui \cite{Boulsane20} improved this bound to 
\begin{equation}
\#\{k : \eps \le \lambda_k \le 1-\eps\} \le \dfrac{\tfrac{1}{\pi^2}\log(2NW) + 0.45 - \tfrac{2}{3}W^2 + \tfrac{\sin^2(2\pi NW)}{6\pi^2N^2}}{\eps(1-\eps)}
\label{eq:Boulsane}
\end{equation}
for all integers $N \ge 1$, $W \in (0,\tfrac{1}{2})$, and $\eps \in (0,\tfrac{1}{2})$. For a fixed $W \in (0,\tfrac{1}{2})$ and large $N$, this bound is roughly half of (\ref{eq:ZhuWakin}). Also, this bound correctly captures the logarithmic dependence on $2NW$ as opposed to just $N$. However, this bound still has a dependence on $\eps$ that is $O(\tfrac{1}{\eps})$. Again, when $N = 1000$, $W = \tfrac{1}{8}$, and $\eps = 10^{-3}$, this result becomes $\#\{k : \eps < \lambda_k < 1-\eps\} \le 1000$. More generally, when $\eps < \tfrac{\log(2NW)}{\pi^2N}$, this bound is worse than the trivial bound of $\#\{k : \eps < \lambda_k < 1-\eps\} \le N$. Boulsane et. al. also used to a result on PSWF eigenvalues by Bonami, Jamming, and Karoui \cite{Bonami18} to prove that the DPSS eigenvalues satisfy 
\begin{equation}
\lambda_k \le 2\exp\left[-\eta\dfrac{k-2NW}{\log(\pi NW)+5}\right] \quad \text{for} \quad 2NW+\log(\pi NW)+6 \le k \le \pi NW, 
\label{eq:Boulsane1}
\end{equation} 
where $\eta = 0.069$, and 
\begin{equation}
\lambda_k \le 2\exp\left[-(2k+1)\log\left(\dfrac{2k+2}{e\pi NW}\right)\right] \quad \text{for} \quad 2 \le \dfrac{e\pi}{2}NW \le k \le N-1.
\label{eq:Boulsane2}
\end{equation} 
However, with no similar lower bounds on the DPSS eigenvalues $\lambda_k$ for $k < 2NW$, they were unable to obtain a bound on $\#\{k : \eps < \lambda_k < 1-\eps\}$ which has a logarithmic dependence on $\eps$. 

In \cite{KarnikFST}, the authors of this paper along with Zhu and Wakin proved that 
\begin{equation}
\#\{k : \eps < \lambda_k < 1-\eps\} \le \left(\dfrac{8}{\pi^2}\log(8N)+12\right)\log\left(\dfrac{15}{\eps}\right)
\label{eq:FST}
\end{equation}
for all $N \in \N$, $W \in (0,\tfrac{1}{2})$, and $\eps \in (0,\tfrac{1}{2})$.  This bound correctly captures the logarithmic dependence on both $N$ and $\eps$, but not the dependence on $W$. Also, the leading constant $\tfrac{8}{\pi^2}$ is four times larger than that of the asymptotic results by Slepian. For comparison with the previous two non-asymptotic bounds, when $N = 1000$, $W = \tfrac{1}{8}$, and $\eps = 10^{-3}$, then this result becomes $\#\{k : \eps < \lambda_k < 1-\eps\} \le 185$, which is still much larger than the actual value of $\#\{k : \eps < \lambda_k < 1-\eps\} = 12$. 

\subsection{Prior results on PSWF eigenvalues}
\label{sec:PrevPSWFBounds}
The PSWF eigenvalues $\widetilde{\lambda}_k$ have a similar behavior as the DPSS eigenvalues. The PSWF eigenvalues are also all strictly between $0$ and $1$, and they have a sum of $\sum_{k = 0}^{\infty}\widetilde{\lambda}_k = \tfrac{2c}{\pi}$, which is the analogous time-bandwidth product in the continuous case. Furthermore, for any $\eps \in (0,\tfrac{1}{2})$, slightly fewer than $\tfrac{2c}{\pi}$ eigenvalues lie in $[1-\eps,1)$, very few eigenvalues lie in $(\eps,1-\eps)$, and the rest lie in $(0,\eps]$. 

Landau and Widom \cite{LandauWidom80} showed that for any fixed $\eps \in (0,\tfrac{1}{2})$, $$\#\left\{k : \eps < \widetilde{\lambda}_k < 1-\eps\right\} = \dfrac{2}{\pi^2}\log(c)\log\left(\dfrac{1}{\eps}-1\right)+o(\log(c)) \quad \text{as} \quad c \to \infty.$$ This asymptotic result is similar in form to Slepian's result from \cite{SlepianV}, except the first logarithm contains the time-bandwidth product instead of the time duration. Also, this result specifies that the error term scales like $o(\log(c))$. 

Osipov \cite{Osipov13} showed that for any numbers $c > 22$ and $3 < \delta < \tfrac{\pi c}{16}$, the PSWF eigenvalues satisfy $$\widetilde{\lambda}_k < \dfrac{7056^2c^3}{2\pi}\exp\left[-2\delta\left(1-\dfrac{\delta}{2\pi c}\right)\right] \quad \text{for all} \quad k \ge \dfrac{2c}{\pi}+\dfrac{2}{\pi^2}\delta\log\left(\dfrac{4\pi ec}{\delta}\right).$$ With an appropriate choice of $\delta$, this result implies that $$\#\left\{k : \widetilde{\lambda}_k > \eps\right\} \le \dfrac{2c}{\pi}+\dfrac{32}{31\pi^2}\log\left(\dfrac{7056^2c^3}{2\pi \eps}\right)\log\left(\dfrac{31\pi ec}{4\log\left(\tfrac{7056^2c^3}{2\pi \eps}\right)}\right) \quad \text{for} \quad \dfrac{7056^2c^3}{2\pi}\exp\left(-\dfrac{31\pi c}{256}\right) < \eps < \dfrac{1}{2}.$$ Since Landau \cite{Landau93} showed that $\widetilde{\lambda}_{\floor{2c/\pi}-1} \ge \tfrac{1}{2}$, the above result by Osipov shows that $\#\left\{k : \eps < \widetilde{\lambda}_k \le \tfrac{1}{2}\right\} \le O(\log^2(c)+\log(c)\log(\tfrac{1}{\eps}))$. This only bounds roughly half of the transition region. Also, the bound has a suboptimal dependence on $c$, and the constants are rather large.

Israel \cite{Israel15} showed that for any $\alpha \in (0,\tfrac{1}{2}]$, there exists a constant $C_{\alpha} \ge 1$ such that $$\#\left\{k : \eps < \widetilde{\lambda}_k < 1-\eps\right\} \le C_{\alpha}\log^{1+\alpha}\left(\dfrac{\log\left(\tfrac{2c}{\pi}\right)}{\eps}\right)\log\left(\dfrac{2c}{\pi\eps}\right)$$ for all $c \ge \pi$ and $\eps \in (0,\tfrac{1}{2})$. When compared to the bound by Osipov, this bound has an improved dependence on $c$, but a worse dependence on $\eps$. Also, the constant $C_{\alpha}$ is not explicitly given. 

Bonami, Jamming, and Karoui \cite{Bonami18} established the following bounds on PSWF eigenvalues 
\\
$$\widetilde{\lambda}_k \ge 1-\dfrac{7}{\sqrt{c}}\dfrac{(2c)^k}{k!}e^{-c} \quad \text{for} \quad c > 0 \ \text{and} \ 0 \le k < \dfrac{2c}{2.7}$$

$$\widetilde{\lambda}_k \le \dfrac{1}{2}\exp\left[-\eta\dfrac{k-\tfrac{2c}{\pi}}{\log(c)+5}\right] \quad \text{for} \quad c \ge 22 \ \text{and} \ \dfrac{2c}{\pi}+\log(c)+6 \le k \le c$$

$$\widetilde{\lambda}_k \le \exp\left[-(2k+1)\log\left(\dfrac{2k+2}{ec}\right)\right] \quad \text{for} \quad c > 0 \ \text{and} \ k \ge \max\left(\dfrac{ec}{2},2\right),$$ 
\\
where $\eta = 0.069$. The first bound shows that $\widetilde{\lambda}_k$ approaches $1$ as $k$ decreases at a faster than exponential rate.  The second bound shows that $\widetilde{\lambda}_k$ decreases exponentially over a bounded range of values of $k$ that are slightly greater than the time-bandwidth product $\tfrac{2c}{\pi}$. The third bound shows that $\widetilde{\lambda}_k$ approaches $0$ a faster than exponential rate once $k > \max(\tfrac{ec}{2},2)$. Unfortunately, one can check that for any $c > 0$ and any integer $0.43c \le k < \tfrac{2c}{2.7}$, the first bound is negative, and thus, uninformative. Thus, these bounds give no information about $\widetilde{\lambda}_k$ for $k \in [0.43c,\tfrac{2c}{\pi}+\log(c)+6) \cup (c,\tfrac{ec}{2})$, which is a total of $O(c)$ values of $k$.

\subsection{New results}
\label{sec:NewResults}
In this paper, we will establish the following two non-asymptotic bounds on the number of DPSS eigenvalues in the transition region $(\eps,1-\eps)$.

\begin{theorem}
\label{thm:BoundWithoutW}
For any $N \in \N$, $W \in (0,\tfrac{1}{2})$, and $\eps \in (0,\tfrac{1}{2})$,  $$\#\{k : \eps < \lambda_k < 1-\eps \} \le 2\ceil{\dfrac{1}{\pi^2}\log(4N)\log\left(\dfrac{4}{\eps(1-\eps)}\right)}.$$
\end{theorem}

\begin{theorem}
\label{thm:BoundWithW}
For any $N \in \N$, $W \in (0,\tfrac{1}{2})$, and $\eps \in (0,\tfrac{1}{2})$,  $$\#\{k : \eps < \lambda_k < 1-\eps\} \le \dfrac{2}{\pi^2}\log(100NW+25)\log\left(\dfrac{5}{\eps(1-\eps)}\right) + 7.$$
\end{theorem}

Both Theorem~\ref{thm:BoundWithoutW} and Theorem~\ref{thm:BoundWithW} capture the logarithmic dependence of the width of the transition region on $N$ and $\eps$. Also, both bounds have a leading constant of $\tfrac{2}{\pi^2}$, which is consistent with the asymptotic result by Slepian. Furthermore, Theorem~\ref{thm:BoundWithW} also captures the logarithmic dependence on $W$. We choose to include Theorem~\ref{thm:BoundWithoutW} since the proof is much simpler and since the bound in Theorem~\ref{thm:BoundWithoutW} is better than the bound in Theorem~\ref{thm:BoundWithW} when $W \ge \tfrac{1}{25}$. Again for comparison with the existing non-asymptotic bounds, when $N = 1000$, $W = \tfrac{1}{8}$, and $\eps = 10^{-3}$, the bound in Theorem~\ref{thm:BoundWithoutW} becomes $\#\{k : \eps < \lambda_k < 1-\eps\} \le 14$ and the bound in Theorem~\ref{thm:BoundWithW} becomes $\#\{k : \eps < \lambda_k < 1-\eps\} \le 23$. Both of these bounds are much closer to the actual value $\#\{k : \eps < \lambda_k < 1-\eps\} = 12$ than any of the existing non-asymptotic bounds.

With the non-asymptotic bounds in Theorem~\ref{thm:BoundWithoutW} and Theorem~\ref{thm:BoundWithW}, the following bounds on the eigenvalues themselves are almost immediate. 
\begin{cor}
\label{cor:DPSSEigBounds}
For any $N \in \N$, $W \in (0,\tfrac{1}{2})$, we have 
$$\lambda_k \ge 1-\min\left\{8\exp\left[-\dfrac{\floor{2NW}-k-2}{\tfrac{2}{\pi^2}\log(4N)}\right],10\exp\left[-\dfrac{\floor{2NW}-k-7}{\tfrac{2}{\pi^2}\log(100NW+25)}\right]\right\} \quad \text{for} \quad 0 \le k \le \floor{2NW}-1$$
and
$$\lambda_k \le \min\left\{8\exp\left[-\dfrac{k-\ceil{2NW}-1}{\tfrac{2}{\pi^2}\log(4N)}\right],10\exp\left[-\dfrac{k-\ceil{2NW}-6}{\tfrac{2}{\pi^2}\log(100NW+25)}\right]\right\} \quad \text{for} \quad \ceil{2NW} \le k \le N-1.$$
\end{cor}

To the best of our knowledge, the first bound in Corollary~\ref{cor:DPSSEigBounds} represents the first non-asymptotic lower bound on the DPSS eigenvalues $\lambda_k$ for $k \le \floor{2NW}-1$. 

The second bound in Corollary~\ref{cor:DPSSEigBounds} is similar in form to (\ref{eq:Boulsane1}), except this result has an exponential decay rate of $\tfrac{\pi^2}{2\log(100 NW+25)}$ instead of $\tfrac{0.069}{\log(\pi NW)+5}$. It is not hard to check that when $NW \ge 0.07$, the exponential decay rate from the second bound is at least $71$ times larger than the exponential decay rate in (\ref{eq:Boulsane1}). 

The second bound in Corollary~\ref{cor:DPSSEigBounds} does not capture the faster than exponential decay rate of $\lambda_k$ for $k \ge \tfrac{e\pi}{2}NW$ that is demonstrated by (\ref{eq:Boulsane2}). However, we note that for $NW \ge 15$, the second bound in Corollary~\ref{cor:DPSSEigBounds} will yield an upper bound for $\lambda_{\ceil{e\pi NW/2}}$ that is less than the single precision machine epsilon of $2^{-23} \approx 1.2\times 10^{-7}$. Also, for $NW \ge 31$, the second bound in Corollary~\ref{cor:DPSSEigBounds} will yield an upper bound for $\lambda_{\ceil{e\pi NW/2}}$ that is below the double precision machine epsilon of $2^{-52} \approx 2.2\times 10^{-16}$. So these bounds are still useful. 

From the eigenvalue bounds in Corollary~\ref{cor:DPSSEigBounds}, we can obtain the following bounds on the sums of the leading and trailing DPSS eigenvalues.
\begin{cor}
\label{cor:DPSSEigSumBounds}
For any $N \in \N$, $W \in (0,\tfrac{1}{2})$, we have 
$$\sum_{k = 0}^{K-1}(1-\lambda_k) \le \min\left\{\dfrac{16}{\pi^2}\log(4N)\exp\left[-\dfrac{\floor{2NW}-K-2}{\tfrac{2}{\pi^2}\log(4N)}\right], \dfrac{20}{\pi^2}\log(100NW+25)\exp\left[-\dfrac{\floor{2NW}-K-7}{\tfrac{2}{\pi^2}\log(100NW+25)}\right]\right\}$$ for $1 \le K \le \floor{2NW}$, and
$$\sum_{k = K}^{N-1}\lambda_k \le \min\left\{\dfrac{16}{\pi^2}\log(4N)\exp\left[-\dfrac{K-\ceil{2NW}-2}{\tfrac{2}{\pi^2}\log(4N)}\right], \dfrac{20}{\pi^2}\log(100NW+25)\exp\left[-\dfrac{K-\ceil{2NW}-7}{\tfrac{2}{\pi^2}\log(100NW+25)}\right]\right\}$$ for $\ceil{2NW} \le K \le N-1$.
\end{cor}

Abreu and Romero \cite{Abreu17} have shown that bias of Thomson's multitaper spectral estimate \cite{Thomson82} (using $K$ tapers) is bounded by the sum of two terms, one which is $\lesssim W^2$ and another which is $\lesssim \tfrac{1}{K}\sum_{k = 0}^{K-1}(1-\lambda_k)$. The variance of this estimate is $\lesssim \tfrac{1}{K}$. By choosing the number of tapers to be $K = 2NW - O(\log(NW)\log(\tfrac{\log(NW)}{\delta}))$ for some small $\delta > 0$ instead of the typically used $\floor{2NW}-1$, one can significantly reduce the bias while only moderately increasing the variance. 

Davenport and Wakin \cite{DavenportWakin12} have shown that sampled sinusoids with frequency $|f| \le W$ can be represented by a linear combination of the first $K$ Slepian basis vectors with an average representation error of $\tfrac{1}{2W}\sum_{k = K}^{N-1}\lambda_k$. They present asymptotic bounds on this error which hold when $K = 2NW(1+\rho)$ Slepian basis vectors are used for some fixed constant $\rho > 0$. The second part of Corollary~\ref{cor:DPSSEigSumBounds} shows that we can obtain a small average representation error by using only $K = 2NW + O(\log(NW)\log(\tfrac{\log(NW)}{\delta}))$ Slepian basis vectors for some small $\delta > 0$.

Note that all of the above bounds which depend on $W$ can be easily refined if $W$ is close to $\tfrac{1}{2}$. Let $\mB' \in \R^{N \times N}$ be the prolate matrix with bandwidth parameter $\tfrac{1}{2}-W$, i.e. $\mB'[m,n] = \tfrac{\sin[2\pi(1/2-W)(m-n)]}{\pi(m-n)}$. One can check that $\mB' = \mD(\mId-\mB)\mD^*$ where $\mD \in \R^{N \times N}$ is a unitary diagonal matrix with entries $\mD[n,n] = (-1)^n$. Hence, the eigenvalues of $\mB'$ and $\mB$ are related by $\lambda_k(N,\tfrac{1}{2}-W) = 1-\lambda_{N-1-k}(N,W)$ for $k = 0,\ldots,N-1$. Therefore, for any $N \in \N$ and $\eps \in (0,\tfrac{1}{2})$, the width of the transition region is the same for bandwidths of $W$ and $\tfrac{1}{2}-W$. In particular, when $W \in (\tfrac{1}{4},\tfrac{1}{2})$, we can obtain stronger bounds on both $\#\{k : \eps < \lambda_k(N,W) < 1-\eps\}$ and $\lambda_k(N,W)$ by instead considering the bounds on $\#\{k : \eps < \lambda_k(N,\tfrac{1}{2}-W) < 1-\eps\}$ and $\lambda_k(N,\tfrac{1}{2}-W)$. 

Finally, we note that a result by Boulsane, Bourguiba, and Karoui \cite{Boulsane20} shows that the DPSS eigenvalues and the PSWF eigenvalues satisfy $\lambda_k(N,W) \to \widetilde{\lambda}_k(\pi NW)$ as $N \to \infty$ and $W \to 0^+$ with $NW$ held constant. Intuitively, this means that the continuous energy concentration problem is the limit of the discrete energy concentration problem as the discritization gets arbitrarily fine. With this result, the above non-asymptotic results on the DPSS eigenvalues which depend on $NW$ (and not just $N$) can be used to obtain the following non-asymptotic results on the PSWF eigenvalues.

\begin{theorem}
\label{thm:PSWFEigGap}
For any $c > 0$ and $\eps \in (0,\tfrac{1}{2})$, $$\#\left\{k : \eps < \widetilde{\lambda}_k < 1-\eps\right\} \le \dfrac{2}{\pi^2}\log\left(\dfrac{100c}{\pi}+25\right)\log\left(\dfrac{5}{\eps(1-\eps)}\right)+7.$$
\end{theorem}

\begin{cor}
\label{cor:PSWFEigBounds}
For any $c > 0$ and $\eps \in (0,\tfrac{1}{2})$, $$\widetilde{\lambda}_k \ge 1-10\exp\left[-\dfrac{\floor{\tfrac{2c}{\pi}}-k-7}{\tfrac{2}{\pi^2}\log\left(\tfrac{100c}{\pi}+25\right)}\right] \quad \text{for} \quad 0 \le k \le \floor{\dfrac{2c}{\pi}}-1,$$ and $$\widetilde{\lambda}_k \le 10\exp\left[-\dfrac{k-\ceil{\tfrac{2c}{\pi}}-6}{\tfrac{2}{\pi^2}\log\left(\tfrac{100c}{\pi}+25\right)}\right] \quad \text{for} \quad k \ge \ceil{\dfrac{2c}{\pi}}.$$
\end{cor}

\begin{cor}
\label{cor:PSWFEigSumBounds}
For any $c > 0$, we have 
$$\sum_{k = 0}^{K-1}\left(1-\widetilde{\lambda}_k\right) \le \dfrac{20}{\pi^2}\log\left(\dfrac{100c}{\pi}+25\right)\exp\left[-\dfrac{\floor{\tfrac{2c}{\pi}}-K-7}{\tfrac{2}{\pi^2}\log(\tfrac{100c}{\pi}+25)}\right] \quad \text{for} \quad 1 \le K \le \floor{\dfrac{2c}{\pi}},$$ and
$$\sum_{k = K}^{\infty}\widetilde{\lambda}_k \le \dfrac{20}{\pi^2}\log\left(\dfrac{100c}{\pi}+25\right)\exp\left[-\dfrac{K-\ceil{\tfrac{2c}{\pi}}-7}{\tfrac{2}{\pi^2}\log(\tfrac{100c}{\pi}+25)}\right] \quad \text{for} \quad K \ge \ceil{\dfrac{2c}{\pi}}.$$
\end{cor}

We note that the non-asymptotic bound in Theorem~\ref{thm:PSWFEigGap} is similar in form to the asymptotic bound by Landau and Widom \cite{LandauWidom80} in that it scales like $O(\log(c) \log(\tfrac{1}{\eps}))$ and has the correct leading constant $\tfrac{2}{\pi^2}$. Furthermore, the constants in this bound are all specified and are mild. As such, this is a substantial improvement over the existing non-asymptotic bounds. 

The bounds in Corollary~\ref{cor:PSWFEigBounds} show that the PSWF eigenvalues $\widetilde{\lambda}_k$ approach $1$ and $0$ exponentially as $k$ moves away from the time-bandwidth product $\tfrac{2c}{\pi}$. Also, these bounds are useful (i.e. something stronger than $0 < \widetilde{\lambda}_k < 1$) for all but $O(\log(c))$ values of $k$.

%% file: MainProof.tex
\section{Proof of Bounds on the Number of DPSS Eigenvalues in the Transition Region (Theorems~\ref{thm:BoundWithoutW} and \ref{thm:BoundWithW})}
\label{sec:MainProof}

\subsection{Proof overview}
For any rectangular matrix $\mX \in \C^{M \times N}$. we use the notation $\sigma_k(\mX)$ to denote the $k^{\text{th}}$ largest singular value of $\mX$. If $k > \min\{M,N\}$, we define $\sigma_k(\mX) = 0$. Also, for a Hermitian matrix $\mA \in \C^{N \times N}$, we use the notation $\mu_k(\mA)$ to denote $k^{\text{th}}$ largest eigenvalue of a symmetric matrix. Again, if $k > N$, we define $\mu_k(\mA) = 0$. To be consistent with standard notation, we define $\lambda_k = \mu_{k+1}(\mB)$ for $k \in \{0,\ldots,N-1\}$, i.e. $\lambda_k$ is the $(k+1)^{\text{th}}$ largest eigenvalue of the $N \times N$ prolate matrix $\mB$ with bandwidth parameter $W$, which is defined in (\ref{eq:Prolate}). Although both $\mB$ and $\lambda_k$ depend on $N$ and $W$, our notation will omit this dependence for convenience. 

We first sketch a non-rigorous outline of our proof. We aim to show that $\mB-\mB^2$ has a low numerical rank. Therefore, very few of the eigenvalues of $\mB-\mB^2$ are not near $0$, and thus, very few of the eigenvalues of $\mB$ are not near $1$ or $0$. To show $\mB-\mB^2$ has a low numerical rank, recall that $\mB$ is a matrix representation of the self-adjoint operator $\calT_N\calB_W\calT_N$. Hence, $\mB-\mB^2$ is a matrix representation of the operator \begin{align*}&\calT_N\calB_W\calT_N-(\calT_N\calB_W\calT_N)^2 = \calT_N\calB_W\calB_W\calT_N-\calT_N\calB_W\calT_N\calB_W\calT_N \\ &= \calT_N\calB_W(\calI-\calT_N)\calB_W\calT_N = \calT_N\calB_W(\calI-\calT_N)(\calI-\calT_N)\calB_W\calT_N\end{align*} Thus showing that $\mB-\mB^2$ has a low numerical rank is equivalent to showing that $(\calI-\calT_N)\calB_W\calT_N$ has a low numerical rank. This operator satisfies $$\inner{\delta_{\ell},(\calI-\calT_N)\calB_W\calT_N\delta_n} = \dfrac{\sin[2\pi W(\ell-n)]}{\pi(\ell-n)} \quad \text{for} \quad \ell \in \Z\setminus\{0,\ldots,N-1\}, n \in \{0,\ldots,N-1\},$$ where $\{\delta_n\}_{n \in \Z}$ are the Euclidean basis sequences for $\ell_2(\Z)$. If we let $\mX$ be the ``matrix'' representation of this operator with respect to the Euclidean basis, then the entries of $\mX$ are a smooth function of the row and column indices, and $\mX$ has what is known as a low displacement-rank structure. These facts can be exploited to show that $\mX$, and thus also $(\calI-\calT_N)\calB_W\calT_N$,  has a low numerical rank. Proving this rigorously requires that we truncate $(\calI-\calT_N)\calB_W\calT_N$ to a finite dimensional subspace, and take the limit as the dimension goes to infinity. 

The following lemma allows us to start a formal and rigorous version of the above argument.
\begin{lemma}
\label{lem:Main}
Suppose for some $r \in \{0,\ldots,N-1\}$ and $L_0 \in \N$, there exists a sequence of matrices $\mX_L \in \R^{2L \times N}$ for $L = L_0, L_0+1, \ldots, $ such that:
\begin{itemize}
\item $\displaystyle\lim_{L \to \infty}\left\|(\mB-\mB^2) - \mX_L^*\mX_L\right\|_F^2 = 0$,

\item $\sigma_{r+1}(\mX_L) \le \sqrt{\eps(1-\eps)}$ for all $L \ge L_0$.
\end{itemize}
Then, $\#\{k : \eps < \lambda_k < 1-\eps\} \le r$.
\end{lemma}

\begin{proof}
By the first property, $\displaystyle\lim_{L \to \infty}\mu_{r+1}(\mX_L^*\mX_L)$ exists and is equal to $\mu_{r+1}(\mB-\mB^2)$. Then by using the fact that $\mu_{r+1}(\mX_L^*\mX_L) = \sigma_{r+1}(\mX_L)^2$ for all $L \ge L_0$ along with the second property, we have $$\mu_{r+1}(\mB-\mB^2) = \displaystyle\lim_{L \to \infty}\mu_{r+1}(\mX_L^*\mX_L) = \displaystyle\lim_{L \to \infty}\sigma_{r+1}(\mX_L)^2 \le \eps(1-\eps).$$ The eigenvalues of $\mB-\mB^2$ are $\{\lambda_k(1-\lambda_k)\}_{k = 0}^{N-1}$. Also, the function $\lambda \mapsto  \lambda(1-\lambda)$ is increasing for $\lambda < \tfrac{1}{2}$ and decreasing for $\lambda > \tfrac{1}{2}$ and symmetric about $\lambda = \tfrac{1}{2}$. As a result, $\eps < \lambda < 1-\eps$ if and only if $\lambda(1-\lambda) > \eps(1-\eps)$. Therefore, $$\#\{k : \eps < \lambda_k < 1-\eps\} = \#\{k : \lambda_k(1-\lambda_k) > \eps(1-\eps)\} = \#\{k : \mu_k(\mB-\mB^2) > \eps(1-\eps)\} \le r.$$
\end{proof}

We will prove both Theorem~\ref{thm:BoundWithoutW} and Theorem~\ref{thm:BoundWithW} by using Lemma~\ref{lem:Main}. In Section~\ref{sec:ConstructXL}, we construct a sequence of matrices $\mX_L \in \R^{2L \times N}$ which satisfies the first property above. In \cref{sec:ProofThm1}, we show that the singular values of each matrix $\mX_L$ decays exponentially, which allows us to obtain the bound in Theorem~\ref{thm:BoundWithoutW}. In Section~\ref{sec:ProofThm2}, we refine the rate at which the singular values of each $\mX_L$ decay, which allows us to obtain the bound in Theorem~\ref{thm:BoundWithW}.

\subsection{Constructing the sequence of matrices $\mX_L$}
\label{sec:ConstructXL}

First, we state an identity involving the sinc function. For any $W \in (0,\tfrac{1}{2})$ and any $m,n \in \Z$, $$\sum_{\ell = -\infty}^{\infty}\dfrac{\sin[2\pi W(\ell-m)]}{\pi(\ell-m)}\dfrac{\sin[2\pi W(\ell-n)]}{\pi(\ell-n)} = \dfrac{\sin[2\pi W(m-n)]}{\pi(m-n)}.$$ By using this identity, we can write the entries of $\mB-\mB^2$ as: 
\begin{align}
(\mB-\mB^2)[m,n] &= \mB[m,n] - \sum_{\ell = 0}^{N-1}\mB[m,\ell]\mB[\ell,n] \notag
\\
&= \dfrac{\sin[2\pi W(m-n)]}{\pi(m-n)} - \sum_{\ell = 0}^{N-1}\dfrac{\sin[2\pi W(m-\ell)]}{\pi(m-\ell)}\dfrac{\sin[2\pi W(\ell-n)]}{\pi(\ell-n)} \notag
\\
&= \sum_{\ell = -\infty}^{\infty}\dfrac{\sin[2\pi W(\ell-m)]}{\pi(\ell-m)}\dfrac{\sin[2\pi W(\ell-n)]}{\pi(\ell-n)} - \sum_{\ell = 0}^{N-1}\dfrac{\sin[2\pi W(\ell-m)]}{\pi(\ell-m)}\dfrac{\sin[2\pi W(\ell-n)]}{\pi(\ell-n)} \notag
\\
&= \sum_{\ell = -\infty}^{-1}\dfrac{\sin[2\pi W(\ell-m)]}{\pi(\ell-m)}\dfrac{\sin[2\pi W(\ell-n)]}{\pi(\ell-n)} + \sum_{\ell = N}^{\infty}\dfrac{\sin[2\pi W(\ell-m)]}{\pi(\ell-m)}\dfrac{\sin[2\pi W(\ell-n)]}{\pi(\ell-n)} \label{eq:BminusB2}
\end{align}
where the rearranging of terms is valid since the summands decay like $O(|\ell|^{-2})$ as $\ell \to \pm \infty$, and thus, all the sums are absolutely convergent. 

For each integer $L \ge 1$, we define an index set $$\setI_L = \{-L,-L+1\ldots,-2,-1\}\cup\{N,N+1,\ldots,N+L-2,N+L-1\}$$ and we define $\mX_L \in \R^{2L \times N}$ by $$\mX_L[\ell,n] = \dfrac{\sin[2\pi W(\ell-n)]}{\pi(\ell-n)} \quad \text{for} \quad \ell \in \setI_L \ \text{and} \ n \in \{0,\ldots,N-1\}.$$ Note that we index the rows of $\mX_L$ by $\setI_L$ instead of the usual $0,1,\ldots,2L-1$ for convenience. We will also index the rows and/or columns of other matrices with dimension $2L$ by $\setI_L$. With this definition, the entries of $\mX_L^*\mX_L$ are 
\begin{align}
(\mX_L^*\mX_L)[m,n] &= \sum_{\ell \in \setI_L}\mX_L[\ell,m]\mX_L[\ell,n] \notag
\\
&= \sum_{\ell = -L}^{-1}\dfrac{\sin[2\pi W(\ell-m)]}{\pi(\ell-m)}\dfrac{\sin[2\pi W(\ell-n)]}{\pi(\ell-n)} + \sum_{\ell = N}^{N+L-1}\dfrac{\sin[2\pi W(\ell-m)]}{\pi(\ell-m)}\dfrac{\sin[2\pi W(\ell-n)]}{\pi(\ell-n)} \label{eq:XX}
\end{align}

From Equations~\ref{eq:BminusB2} and \ref{eq:XX} above, $\displaystyle\lim_{L \to \infty}(\mX_L^*\mX_L)[m,n] = (\mB-\mB^2)[m,n]$ for each of the $N^2$ entries. Therefore, $$\displaystyle\lim_{L \to \infty}\left\|(\mB-\mB^2) - \mX_L^*\mX_L\right\|_F^2 = \lim_{L \to \infty} \sum_{m = 0}^{N-1}\sum_{n = 0}^{N-1}\left|(\mB-\mB^2)[m,n] - (\mX_L^*\mX_L)[m,n]\right|^2 = 0.$$ This shows that the sequence of matrices $\mX_L \in \R^{2L \times N}$ satisfies the first property of Lemma~\ref{lem:Main}. We will now focus on bounding the singular values of each $\mX_L$ in order to prove that these matrices $\mX_L$ satisfy the second property of Lemma~\ref{lem:Main}.

\subsection{Proof of Theorem~\ref{thm:BoundWithoutW}}
\label{sec:ProofThm1}
In \cite{BT19}, Beckermann and Townsend showed that matrices with a low-rank displacement have rapidly decaying singular values provided the spectra of the left and right displacement matrices are well separated. More specifically, suppose a matrix $\mX \in \C^{M \times N}$ satisfies the displacement equation $$\mC\mX-\mX\mD = \mU\mV^*$$ where $\mC \in \C^{M \times M}$ and $\mD \in \C^{N \times N}$ are normal matrices and $\mU \in \C^{M \times \nu}$ and $\mD \in \C^{N \times \nu}$. If there are closed, disjoint subsets $E,F$ of $\C$ such that $\Spec(\mC) \subset E$ and $\Spec(\mD) \subset F$ then the singular values of $\mX$ satisfy $$\sigma_{\nu k+1}(\mX) \le \sigma_1(\mX)Z_k(E,F)$$ for all integers $k \ge 0$, where $Z_k(E,F)$ are the Zolotarev numbers \cite{Zolotarev1877} for the sets $E$ and $F$. As a rule of thumb, when $E$ and $F$ are ``well-separated'', $Z_k(E,F)$ decays exponentially with $k$. For more details about Zolotarev numbers, see Appendix~\ref{sec:Zolotarev}.

In Appendix~\ref{sec:LowRankDisplacement}, we build on the work of Beckermann and Townsend to prove the following theorem.
\begin{theorem}
\label{thm:UnboundedSpec}
Suppose $\mX \in \C^{M \times N}$ satisfies the displacement equation $$\mC\mX-\mX\mD = \mU\mV^*$$ where $\mC \in \C^{M \times M}$ and $\mD \in \C^{N \times N}$ are normal matrices with real eigenvalues and $\mU \in \C^{M \times \nu}$ and $\mV \in \C^{N \times \nu}$. If $\text{Spec}(\mC) \subset (-\infty,c_1]\cup[c_2,\infty)$ and $\text{Spec}(\mD) \subset [d_1,d_2]$ where $c_1 < d_1 < d_2 < c_2$, then for any integer $k \ge 0$, $$\sigma_{\nu k+1}(\mX) \le 4\|\mX\|\exp\left[-\dfrac{\pi^2 k}{\log(16\gamma)}\right] \quad \text{where} \quad \gamma = \dfrac{(c_2-d_1)(d_2-c_1)}{(c_2-d_2)(d_1-c_1)}.$$
\end{theorem}

We now show that the matrices $\mX_L$ defined in Section~\ref{sec:ConstructXL} satisfy a low-rank displacement equation, and use Theorem~\ref{thm:UnboundedSpec} to bound their singular values. Define a diagonal matrix $\mD \in \R^{N \times N}$ by $\mD[n,n] = n$ for $n \in \{0,\ldots,N-1\}$. For each integer $L \ge 1$, define a diagonal matrix $\mC_L \in \R^{2L \times 2L}$ by $\mC_L[\ell,\ell] = \ell$ for $\ell \in \setI_L$ (again, we index $\mC_L$ by $\ell \in \setI_L$ for convenience). With this definition, we have 
\begin{align*}
(\mC_L\mX_L-\mX_L\mD)[\ell,n] &= \mC_L[\ell,\ell]\mX_L[\ell,n] - \mX_L[\ell,n]\mD[n,n]
\\
&= \ell \cdot \dfrac{\sin[2\pi W(\ell-n)]}{\pi(\ell-n)} - \dfrac{\sin[2\pi W(\ell-n)]}{\pi(\ell-n)} \cdot n
\\
&= \dfrac{1}{\pi}\sin[2\pi W(\ell-n)]
\\
&= \dfrac{1}{\pi}\left[\sin(2\pi W\ell)\cos(2\pi Wn) - \cos(2\pi W\ell)\sin(2\pi Wn)\right].
\end{align*}

From this, it is clear that we can factor $$\mC_L\mX_L-\mX_L\mD = \mU_L\mV^*$$ where $\mU_L \in \R^{2L \times 2}$ is defined by $$\mU_L[\ell,0] = \dfrac{1}{\sqrt{\pi}}\sin(2\pi W\ell) \quad \text{and} \quad \mU_L[\ell,1] = \dfrac{1}{\sqrt{\pi}}\cos(2\pi W\ell)\quad \text{for} \quad\ell \in \setI_L,$$ and $\mV \in \R^{N \times 2}$ is defined by $$\mV[n,0] = \dfrac{1}{\sqrt{\pi}}\cos(2\pi Wn) \quad \text{and} \quad \mV[n,1] = -\dfrac{1}{\sqrt{\pi}}\sin(2\pi Wn) \quad \text{for} \quad n \in \{0,\ldots,N-1\}.$$ In other words, $\mX_L$ has a rank-$2$ displacement with respect to the matrices $\mC_L$ and $\mD$. 

Since $\Spec(\mC_L) = \setI_L \subset (-\infty,-1] \cup [N,\infty)$ and $\Spec(\mD) = \{0,\ldots,N-1\} \subset [0,N-1],$ we can apply Theorem~\ref{thm:UnboundedSpec} with parameters $c_1 = -1$, $d_1 = 0$, $d_2 = N-1$, $c_2 = N$, and $\nu = 2$. The theorem tells us that for every integer $k \ge 0$, $$\sigma_{2k+1}(\mX_L) \le 4\|\mX_L\|\exp\left[-\dfrac{\pi^2 k}{\log(16\gamma)}\right] \quad \text{where} \quad \gamma = \dfrac{(c_2-d_1)(d_2-c_1)}{(c_2-d_2)(d_1-c_1)} = N^2.$$

For any $L \ge 1$, $\mX_L$ is a submatrix of $\mX_{L+1}$, and so, $\mX_L^*\mX_L \preceq \mX_{L+1}^*\mX_{L+1}$. Hence, $\mX_L^*\mX_L \preceq \displaystyle\lim_{L \to \infty}\mX_L^*\mX_L = \mB-\mB^2$. Therefore, $$\|\mX_L\|^2 = \|\mX_L^*\mX_L\| \le \|\mB-\mB^2\| = \displaystyle\max_{k}\left[\lambda_k-\lambda_k^2\right] \le \max_{0 \le \lambda \le 1}[\lambda-\lambda^2] = \dfrac{1}{4},$$ and thus, $\|\mX_L\| \le \tfrac{1}{2}$ for all $L \ge 1$. Substituting $\gamma = N^2$ and $\|\mX_L\| \le \tfrac{1}{2}$ into the above bound yields $$\sigma_{2k+1}(\mX_L) \le 2\exp\left[-\dfrac{\pi^2 k}{\log(16N^2)}\right]$$
for all integers $k \ge 0$. So, if we set $$k = \ceil{\dfrac{1}{\pi^2}\log(16N^2)\log\left(\dfrac{2}{\sqrt{\eps(1-\eps)}}\right)} = \ceil{\dfrac{1}{\pi^2}\log(4N)\log\left(\dfrac{4}{\eps(1-\eps)}\right)},$$ we obtain $\sigma_{2k+1}(\mX_L) \le \sqrt{\eps(1-\eps)}$ for all $L \ge 1$. 

This proves the second property in Lemma~\ref{lem:Main} for $r = 2k$ and $L_0 = 1$. Therefore, we have proved that $$\#\{k : \eps < \lambda_k < 1-\eps\} \le 2k = 2\ceil{\dfrac{1}{\pi^2}\log(4N)\log\left(\dfrac{4}{\eps(1-\eps)}\right)},$$ which is exactly the content of Theorem~\ref{thm:BoundWithoutW}.

\subsection{Proof of Theorem~\ref{thm:BoundWithW}}
\label{sec:ProofThm2}
First, note that if $W \in [\tfrac{1}{4},\tfrac{1}{2})$, the bound in Theorem~\ref{thm:BoundWithW} is greater than the bound in Theorem~\ref{thm:BoundWithoutW}, which has already been established. So we will henceforth assume that $W \in (0,\tfrac{1}{4})$. 

Now, set $L_1 = \left\lfloor\tfrac{1}{4W}\right\rfloor$ (clearly, $L_1 \ge 1$). For each integer $L \ge L_1+1$, we partition the index set $$\setI_L = \{-L,-L+1\ldots,-2,-1\}\cup\{N,N+1,\ldots,N+L-2,N+L-1\}$$ into three sets 
\begin{align*}
\setI^{(0)}_L &= \{-L,-L+1\ldots,-L_1-2,-L_1-1\}\cup\{N+L_1,N+L_1+1,\ldots,N+L-2,N+L-1\}
\\
\setI^{(1)}_L &= \{-L_1,-L_1+1,\ldots,-2,-1\}
\\
\setI^{(2)}_L &= \{N,N+1,\ldots,N+L_1-2,N+L_1-1\}
\end{align*}
and then accordingly partition $\mX_L$ into three submatrices $\mX^{(0)}_L \in \R^{2(L-L_1) \times N}$, $\mX^{(1)}_L \in \R^{L_1 \times N}$, and $\mX^{(2)}_L \in \R^{L_1 \times N}$ defined by $$\mX^{(i)}_L[\ell,n] = \mX_L[\ell,n] \quad \text{for} \quad \ell \in \setI^{(i)}_L \ \text{and} \ n \in \{0,\ldots,N-1\}.$$ Once again, we index the rows of each $\mX^{(i)}_L$ by $\ell \in \setI^{(i)}_L$ for convenience. We proceed to bound the singular values of $\mX^{(0)}_L, \mX^{(1)}_L, \mX^{(2)}_L$, and then use these bounds to bound the singular values of $\mX_L$.

\subsubsection*{Singular values of $\mX^{(0)}_L$}
The submatrix $\mX^{(0)}_L$, has the same low-rank displacement structure as $\mX_L$. Specifically, we can write $$\mC^{(0)}_L\mX^{(0)}_L-\mX^{(0)}_L\mD = \mU^{(0)}_L\mV^*$$ where $\mC^{(0)}_L \in \R^{2(L-L_1) \times 2(L-L_1)}$ is the diagonal submatrix of $\mC_L$ defined by $\mC^{(0)}_L[\ell,\ell] = \ell$ for $\ell \in \setI^{(0)}_L$, $\mU^{(0)}_L \in \R^{2(L-L_1) \times 2}$ is the submatrix of $\mU_L$ defined in by $\mU^{(0)}_L[\ell,q] = \mU_L[\ell,q]$ for $\ell \in \setI^{(0)}_L$ and $q \in \{0,1\}$, and $\mD$ and $\mV$ are the same as defined in Section~\ref{sec:ProofThm1}. 

Since $\Spec(\mC^{(0)}_L) = \setI^{(0)}_L \subset (-\infty,-L_1-1] \cup [N+L_1,\infty)$ and $\Spec(\mD) = \{0,\ldots,N-1\} \subset [0,N-1],$ we can once again apply Theorem~\ref{thm:UnboundedSpec}, but with the parameters $c_1 = -L_1-1$, $d_1 = 0$, $d_2 = N-1$, $c_2 = N+L_1$, and $\nu = 2$. Then, the theorem tells us that for any integer $k_0 \ge 0$, $$\sigma_{2k_0+1}(\mX^{(0)}) \le 4\|\mX^{(0)}_L\|\exp\left[-\dfrac{\pi^2 k_0}{\log(16\gamma)}\right] \quad \text{where} \quad \gamma = \dfrac{(c_2-d_1)(d_2-c_1)}{(c_2-d_2)(d_1-c_1)} = \left(\dfrac{N+L_1}{L_1+1}\right)^2.$$ 

Since $\mX^{(0)}_L$ is a submatrix of $\mX_L$, we have $\|\mX^{(0)}_L\| \le \|\mX_L\| \le \tfrac{1}{2}$. Also, since $L_1 = \floor{\tfrac{1}{4W}} \ge \tfrac{1}{4W}-1 > 0$ and $\tfrac{N+x}{x+1}$ is a non-increasing function of $x > 0$, we can bound $$\gamma = \left(\dfrac{N+L_1}{L_1+1}\right)^2 \le \left(\dfrac{N+(\tfrac{1}{4W}-1)}{(\tfrac{1}{4W}-1)+1}\right)^2 = (4NW+1-4W)^2 \le (4NW+1)^2.$$ Substituting $\gamma \le (4NW+1)^2$ and $\|\mX^{(0)}_L\| \le \tfrac{1}{2}$ into the above bound  yields $$\sigma_{2k_0+1}(\mX^{(0)}_L) \le 2\exp\left[-\dfrac{\pi^2 k_0}{\log[16(4NW+1)^2]}\right] = 2\exp\left[-\dfrac{\pi^2 k_0}{2\log(16NW+4)}\right]$$ for all integers $k_0 \ge 0$.

\subsubsection*{Singular values of $\mX^{(1)}_L$} 
To bound the singular values of $\mX^{(1)}_L$, we exploit the fact that its entries $\mX^{(1)}_L[\ell,n] = \dfrac{\sin[2\pi W(\ell-n)]}{\pi(\ell-n)}$ are a smooth function of $\ell$ and $n$ to construct a tunable low-rank approximation of $\mX^{(1)}_L$.

Define the sinc function $g(t) = \dfrac{\sin(2\pi Wt)}{\pi t}$. For each $n = 0,\ldots,N-1$, define $$g_n(t) = g(t-n) = \dfrac{\sin[2\pi W(t-n)]}{\pi(t-n)},$$ and let $$P_{k,n}(t) = \sum_{m = 0}^{k-1}p_{m,n}t^m$$ be the degree $k-1$ Chebyshev interpolating polynomial for $g_n(t)$ on the interval $[-L_1,-1]$. 

We now define the low rank approximation $\widetilde{\mX}^{(1)}_L \in \R^{L_1 \times N}$ by $$\widetilde{\mX}^{(1)}_L[\ell,n] = P_{k,n}(\ell) = \sum_{m = 0}^{k-1}p_{m,n}\ell^m \quad \text{for} \quad \ell \in \{-L_1,\ldots,-1\} \ \text{and} \ n \in \{0,\ldots,N-1\}.$$ We can factor $\widetilde{\mX}^{(1)}_L = \mW\mP$ where $\mW \in \R^{L_1 \times k}$ and $\mP \in \R^{k \times N}$ are defined by $\mW[\ell,m] = \ell^m$ and $\mP[m,n] = p_{m,n}$. Hence, $\rank(\widetilde{\mX}^{(1)}_L) \le k$.

By Theorem~\ref{thm:ChebyshevInterp} (in Appendix~\ref{sec:SincFunction}), the Chebyshev interpolating polynomial satisfies $$\left|g_n(t) - P_{k,n}(t)\right| \le \dfrac{(L_1-1)^k}{2^{2k-1}k!}\max_{\xi \in [-L_1,-1]}\left|g_n^{(k)}(\xi)\right| \quad \text{for all} \quad t \in [-L_1,-1].$$ Also, by Lemma~\ref{lem:SincDerivatives} (in Appendix~\ref{sec:SincFunction}), the derivatives of the unshifted sinc-function $g(t)$ can be bounded by $$\left|g^{(k)}(t)\right| \le (2\pi W)^k\min\left\{\dfrac{2W}{k+1} , \dfrac{2}{\pi|t|} \right\} \quad \text{for all} \quad t \in \R.$$
Hence, for any $\ell \in \setI^{(1)}_L$ and $n \in \{0,\ldots,N-1\}$, we have
\begin{align*}
\left|\mX^{(1)}_L[\ell,n] - \widetilde{\mX}^{(1)}_L[\ell,n]\right| &= \left|g_n(\ell) - P_{k,n}(\ell)\right|
\\
&\le \dfrac{(L_1-1)^k}{2^{2k-1}k!}\max_{\xi \in [-L_1,-1]}\left|g_n^{(k)}(\xi)\right|
\\
&= \dfrac{(L_1-1)^k}{2^{2k-1}k!}\max_{\xi \in [-L_1,-1]}\left|g^{(k)}(\xi-n)\right|
\\
&= \dfrac{(L_1-1)^k}{2^{2k-1}k!}\max_{t \in [-L_1-n,-n-1]}\left|g^{(k)}(t)\right|
\\
&\le \dfrac{(L_1-1)^k}{2^{2k-1}k!}\max_{t \in [-L_1-n,-n-1]}(2\pi W)^k\min\left\{\dfrac{2W}{k+1} , \dfrac{2}{\pi|t|} \right\}
\\
&= \dfrac{(L_1-1)^k}{2^{2k-1}k!}(2\pi W)^k\min\left\{\dfrac{2W}{k+1} , \dfrac{2}{\pi(n+1)} \right\}
\\
&= \dfrac{4(\tfrac{\pi}{2}W(L_1-1))^k}{k!}\min\left\{\dfrac{W}{k+1} , \dfrac{1}{\pi(n+1)} \right\}.
\end{align*}

We proceed to bound the Frobenius norm of $\mX^{(1)}_L - \widetilde{\mX}^{(1)}_L$. Set $N_1 = \floor{\tfrac{k+1}{\pi W}}$. Then, 
\begin{align*}
\left\|\mX^{(1)}_L - \widetilde{\mX}^{(1)}_L\right\|_F^2 &= \sum_{n = 0}^{N-1}\sum_{\ell = -L_1}^{-1}\left|\mX^{(1)}_L[\ell,n] - \widetilde{\mX}^{(1)}_L[\ell,n]\right|^2
\\
&\le \sum_{n = 0}^{N-1}\sum_{\ell = -L_1}^{-1}\dfrac{16(\tfrac{\pi}{2}W(L_1-1))^{2k}}{(k!)^2}\min\left\{\dfrac{W^2}{(k+1)^2} , \dfrac{1}{\pi^2(n+1)^2} \right\}
\\
&= \sum_{n = 0}^{N-1}\dfrac{16L_1(\tfrac{\pi}{2}W(L_1-1))^{2k}}{(k!)^2}\min\left\{\dfrac{W^2}{(k+1)^2} , \dfrac{1}{\pi^2(n+1)^2} \right\}
\\
&\le \sum_{n = 0}^{\infty}\dfrac{16L_1(\tfrac{\pi}{2}W(L_1-1))^{2k}}{(k!)^2}\min\left\{\dfrac{W^2}{(k+1)^2} , \dfrac{1}{\pi^2(n+1)^2} \right\}
\\
&\le \dfrac{16L_1(\tfrac{\pi}{2}W(L_1-1))^{2k}}{(k!)^2}\left[\sum_{n = 0}^{N_1-1}\dfrac{W^2}{(k+1)^2} + \sum_{n = N_1}^{\infty}\dfrac{1}{\pi^2(n+1)^2}\right]
\\
&\le \dfrac{16L_1(\tfrac{\pi}{2}W(L_1-1))^{2k}}{(k!)^2}\left[\dfrac{W^2N_1}{(k+1)^2} + \dfrac{1}{\pi^2N_1}\right],
\end{align*}
where the last line follows from the bound $\sum_{n = N_1}^{\infty}\tfrac{1}{(n+1)^2} \le \tfrac{1}{N_1}$.

We proceed to weaken this result to obtain a more usable upper bound as follows:
\begin{align*}
\left\|\mX^{(1)}_L - \widetilde{\mX}^{(1)}_L\right\|_F^2 &\le \dfrac{16L_1(\tfrac{\pi}{2}W(L_1-1))^{2k}}{(k!)^2}\left[\dfrac{W^2N_1}{(k+1)^2} + \dfrac{1}{\pi^2N_1}\right]
\\
&\le \dfrac{16(L_1-1)(\tfrac{\pi}{2}WL_1)^{2k}}{(k!)^2}\left[\dfrac{W^2N_1}{(k+1)^2} + \dfrac{1}{\pi^2N_1}\right]
\\
&= \dfrac{16(\tfrac{\pi}{2}WL_1)^{2k}}{(k!)^2}\left[\dfrac{W^2(L_1-1)N_1}{(k+1)^2} + \dfrac{L_1-1}{\pi^2N_1}\right]
\\
&\le \dfrac{16(\tfrac{\pi}{2}WL_1)^{2k}}{(k!)^2}\left[\dfrac{W^2N_1L_1}{(k+1)^2} + \dfrac{L_1}{\pi^2(N_1+1)}\right]
\\
&\le \dfrac{16(\tfrac{\pi}{2}W \cdot \tfrac{1}{4W})^{2k}}{(k!)^2}\left[\dfrac{W^2 \cdot \tfrac{k+1}{\pi W} \cdot \tfrac{1}{4W}}{(k+1)^2} + \dfrac{\tfrac{1}{4W}}{\pi^2 \cdot \tfrac{k+1}{\pi W}}\right]
\\
&= \dfrac{8}{\pi(k+1)(k!)^2}\left(\dfrac{\pi}{8}\right)^{2k}
\\
&\le \dfrac{5600}{\pi}\left(\dfrac{\pi}{48}\right)^{2k}.
\end{align*}
The 2nd line follows from the inequalities $(L_1-1)^{2k} \le L_1^{2k}$ and $L_1(L_1-1)^{2k} \le (L_1-1)L_1^{2k}$. The 4th line holds since $\tfrac{L_1 - 1}{N_1} \le \tfrac{L_1}{N_1+1}$ is equivalent to $L_1 \le N_1 + 1$, which is true since $L_1 \le \left\lfloor\tfrac{1}{4W}\right\rfloor \le \tfrac{1}{4W} < \tfrac{k+1}{\pi W} \le \left\lfloor\tfrac{k+1}{\pi W}\right\rfloor+1 = N_1+1$. The last line holds due to the fact that $(k+1)(k!)^2 \ge \tfrac{1}{700}6^{2k}$ for all integers $k \ge 0$.

Since $\rank(\widetilde{\mX}^{(1)}_L) \le k$, we have $\sigma_{k+1}(\widetilde{\mX}^{(1)}_L) = 0$. Hence, we can bound $$\sigma_{k+1}(\mX^{(1)}_L) = \left|\sigma_{k+1}(\mX^{(1)}_L) - \sigma_{k+1}(\widetilde{\mX}^{(1)}_L)\right| \le \left\|\mX^{(1)}_L - \widetilde{\mX}^{(1)}_L\right\| \le \left\|\mX^{(1)}_L - \widetilde{\mX}^{(1)}_L\right\|_F \le \sqrt{\dfrac{5600}{\pi}}\left(\dfrac{\pi}{48}\right)^{k}.$$

\subsubsection*{Singular values of $\mX^{(2)}_L$}
We can exploit the symmetry between $\mX^{(2)}_L$ and $\mX^{(1)}_L$ to show that the singular values of $\widetilde{\mX}^{(2)}_L$ are the same as those of $\widetilde{\mX}^{(1)}_L$. Specifically, for any indices $\ell \in \setI^{(2)}_L$ and $n \in \{0,\ldots,N-1\}$, we have that $N-1-\ell \in \setI^{(1)}_L$ and $N-1-n \in \{0,\dots,N-1\}$, and that $$\mX^{(2)}_L[\ell,n] = \dfrac{\sin[2\pi W(\ell-n)]}{\pi(\ell-n)} = \dfrac{\sin[2\pi W((N-1-\ell)-(N-1-n))]}{\pi((N-1-\ell)-(N-1-n))} = \mX^{(1)}_L[N-1-\ell,N-1-n].$$ Since the singular values of a matrix are invariant under permutations of rows/columns, we have $$\sigma_{k+1}(\mX^{(2)}_L) = \sigma_{k+1}(\mX^{(1)}_L) \le \sqrt{\dfrac{5600}{\pi}}\left(\dfrac{\pi}{48}\right)^{k}$$ for all integers $k \ge 0$.

\subsubsection*{Singular values of $\mX_L$}
Due to the way we partitioned $\mX_L$ into three submatrices, we have $$\mX_L^*\mX_L = \mX^{(0)*}_L\mX^{(0)}_L + \mX^{(1)*}_L\mX^{(1)}_L + \mX^{(2)*}_L\mX^{(2)}_L.$$ By the Weyl eigenvalue inequalities, we have $$\mu_{2k_0+2k+1}(\mX_L^*\mX_L) \le \mu_{2k_0+1}(\mX^{(0)*}_L\mX^{(0)}_L) + \mu_{k+1}(\mX^{(1)*}_L\mX^{(1)}_L) + \mu_{k+1}(\mX^{(2)*}_L\mX^{(2)}_L).$$ Hence, we can bound 
\begin{align*}
\sigma_{2k_0+2k+1}(\mX_L)^2 &= \mu_{2k_0+2k+1}(\mX_L^*\mX_L)
\\
&\le \mu_{2k_0+1}(\mX^{(0)*}_L\mX^{(0)}_L) + \mu_{k+1}(\mX^{(1)*}_L\mX^{(1)}_L) + \mu_{k+1}(\mX^{(2)*}_L\mX^{(2)}_L)
\\
&= \sigma_{2k_0+1}(\mX^{(0)}_L)^2 + \sigma_{k+1}(\mX^{(1)}_L)^2 + \sigma_{k+1}(\mX^{(2)}_L)^2
\\
&\le 4\exp\left[-\dfrac{\pi^2k_0}{\log(16NW+4)}\right] + \dfrac{5600}{\pi}\left(\dfrac{\pi}{48}\right)^{2k} + \dfrac{5600}{\pi}\left(\dfrac{\pi}{48}\right)^{2k}
\\
&= 4\exp\left[-\dfrac{\pi^2k_0}{\log(16NW+4)}\right] + \dfrac{11200}{\pi}\left(\dfrac{\pi}{48}\right)^{2k}
\end{align*}
for any integers $k_0 \ge 0$ and $k \ge 1$. 

If we set $$k_0 = \ceil{\dfrac{1}{\pi^2}\log(16NW+4)\log\left(\dfrac{5}{\eps(1-\eps)}\right)}$$ and $$k = \ceil{\dfrac{1}{2\log(\tfrac{48}{\pi})}\log\left(\dfrac{\tfrac{56000}{\pi}}{\eps(1-\eps)}\right)}$$ then we obtain 
\begin{align*}
\sigma_{2k_0+2k+1}(\mX_L)^2 &\le 4\exp\left[-\dfrac{\pi^2k_0}{\log(16NW+4)}\right] + \dfrac{11200}{\pi}\left(\dfrac{\pi}{48}\right)^{2k}
\\
&\le \dfrac{4\eps(1-\eps)}{5} + \dfrac{\eps(1-\eps)}{5}
\\
&= \eps(1-\eps),
\end{align*}
i.e., $\sigma_{2k_0+2k+1}(\mX_L) \le \sqrt{\eps(1-\eps)}$. Our steps hold for all $L \ge L_1+1$. 

This proves the second property of Lemma~\ref{lem:Main} for $r = 2k_0+2k$ and $L_0 = L_1+1$. Therefore, $$\#\{k : \eps < \lambda_k < 1-\eps\} \le 2k_0+2k = 2\ceil{\dfrac{1}{\pi^2}\log(16NW+4)\log\left(\dfrac{5}{\eps(1-\eps)}\right)} + 2\ceil{\dfrac{1}{2\log(\tfrac{48}{\pi})}\log\left(\dfrac{\tfrac{56000}{\pi}}{\eps(1-\eps)}\right)}.$$

We can loosen this bound to make it more ``user friendly'' as follows:
\begin{align*}
\#\{k : \eps < \lambda_k < 1-\eps\} &\le 2\ceil{\dfrac{1}{\pi^2}\log(16NW+4)\log\left(\dfrac{5}{\eps(1-\eps)}\right)} + 2\ceil{\dfrac{1}{2\log(\tfrac{48}{\pi})}\log\left(\dfrac{\tfrac{56000}{\pi}}{\eps(1-\eps)}\right)}
\\
&\le \dfrac{2}{\pi^2}\log(16NW+4)\log\left(\dfrac{5}{\eps(1-\eps)}\right) + \dfrac{1}{\log(\tfrac{48}{\pi})}\log\left(\dfrac{\tfrac{56000}{\pi}}{\eps(1-\eps)}\right) + 4
\\
&= \dfrac{2}{\pi^2}\log(16NW+4)\log\left(\dfrac{5}{\eps(1-\eps)}\right) + \dfrac{1}{\log(\tfrac{48}{\pi})}\log\left(\dfrac{5}{\eps(1-\eps)}\right) + \dfrac{\log\left(\tfrac{11200}{\pi}\right)}{\log(\tfrac{48}{\pi})} + 4
\\
&= \left(\dfrac{2}{\pi^2}\log(16NW+4) + \dfrac{1}{\log(\tfrac{48}{\pi})}\right)\log\left(\dfrac{5}{\eps(1-\eps)}\right) + \dfrac{\log\left(\tfrac{11200}{\pi}\right)}{\log(\tfrac{48}{\pi})} + 4
\\
&= \dfrac{2}{\pi^2}\log\left(\exp\left(\tfrac{\pi^2}{2\log(\tfrac{48}{\pi})}\right)(16NW+4)\right)\log\left(\dfrac{5}{\eps(1-\eps)}\right) + \dfrac{\log\left(\tfrac{11200}{\pi}\right)}{\log(\tfrac{48}{\pi})} + 4
\\
&\le \dfrac{2}{\pi^2}\log(100NW+25)\log\left(\dfrac{5}{\eps(1-\eps)}\right) + 7,
\end{align*}
which establishes \cref{thm:BoundWithW}.

%% file: DPSSEigenvalueBounds.tex
\section{Proof of DPSS Eigenvalue Bounds (Corollaries~\ref{cor:DPSSEigBounds} and \ref{cor:DPSSEigSumBounds})}
\label{sec:DPSSEigenvalueBounds}
First, we state a result from \cite{ZhuWakin17} which bounds $\lambda_k$ for two values of $k$ near $2NW$.

\begin{lemma}
\label{lem:MiddleEigs}
For any $N \in \N$ and $W \in (0,\tfrac{1}{2})$, $$\lambda_{\floor{2NW}-1} \ge \dfrac{1}{2} \ge \lambda_{\ceil{2NW}}.$$
\end{lemma}

To derive bounds on $\lambda_k$, we will set $\eps$ such that the transition region is too narrow to contain $k$, and thus conclude either $\lambda_k \ge 1-\eps$ (if $k \le \floor{2NW}-1$) or $\lambda_k \le \eps$ (if $k \ge \ceil{2NW}$). To derive bounds on $\sum_{k = 0}^{K-1}(1-\lambda_k)$ and $\sum_{k = K}^{N-1}\lambda_k$, we will simply apply the bounds on $\lambda_k$ and the formula for the sum of a geometric series.

\subsection{Lower bounds on $\lambda_k$ for $k \le \floor{2NW}-1$}
For any integer $k$ such that $0 \le k \le \floor{2NW}-1$, set $$\eps = 8\exp\left[-\dfrac{\floor{2NW}-k-2}{\tfrac{2}{\pi^2}\log(4N)}\right],$$ and suppose for sake of contradiction that $\lambda_k < 1-\eps$. 

By using the assumption $k \le \floor{2NW}-1$ and Lemma~\ref{lem:MiddleEigs}, we have $\tfrac{1}{2} \le \lambda_{\floor{2NW}-1} \le \lambda_k < 1-\eps$, i.e., $\eps < \tfrac{1}{2}$. 
Therefore, $\eps < \tfrac{1}{2} \le \lambda_{\floor{2NW}-1} \le \lambda_k < 1-\eps$, i.e. both $k$ and $\floor{2NW}-1$ are in the transition region $\{k' : \eps < \lambda_{k'} < 1-\eps\}$, and thus, so are all the indices $k'$ between $k$ and $\floor{2NW}-1$. Hence, $$\#\{k' : \eps < \lambda_{k'} < 1-\eps\} \ge \#\{k' : k \le k' \le \floor{2NW}-1\} = \floor{2NW}-k.$$ 
However, since $\eps < \tfrac{1}{2}$, by Theorem~\ref{thm:BoundWithoutW} we have $$\#\{k' : \eps < \lambda_{k'} < 1-\eps\} \le 2\ceil{\dfrac{1}{\pi^2}\log(4N)\log\left(\dfrac{4}{\eps(1-\eps)}\right)} < \dfrac{2}{\pi^2}\log(4N)\log\left(\dfrac{8}{\eps}\right)+2 = \floor{2NW}-k.$$ This is a contradiction. Therefore, $$\lambda_k \ge 1-\eps = 1-8\exp\left[-\dfrac{\floor{2NW}-k-2}{\tfrac{2}{\pi^2}\log(4N)}\right] \quad \text{for} \quad 0 \le k \le \floor{2NW}-1.$$

In a similar manner, we can assume $\lambda_k < 1-\eps$, where $$\eps = 10\exp\left[-\dfrac{\floor{2NW}-k-7}{\tfrac{2}{\pi^2}\log(100NW+25)}\right],$$ and then invoke \cref{thm:BoundWithW} to obtain a contradiction. Therefore, $$\lambda_k \ge 1-\eps = 1-10\exp\left[-\dfrac{\floor{2NW}-k-7}{\tfrac{2}{\pi^2}\log(100NW+25)}\right] \quad \text{for} \quad 0 \le k \le \floor{2NW}-1.$$ Combining these two bounds establishes the first part of Corollary~\ref{cor:DPSSEigBounds}.

\subsection{Upper bounds on $\lambda_k$ for $k \ge \ceil{2NW}$}
For any integer $k$ such that $\ceil{2NW} \le k \le N-1$, set $$\eps = 8\exp\left[-\dfrac{k-\ceil{2NW}-1}{\tfrac{2}{\pi^2}\log(4N)}\right],$$ and suppose for sake of contradiction that $\lambda_k > \eps$. 

By using the assumption $k \ge \ceil{2NW}$ and Lemma~\ref{lem:MiddleEigs}, we have $\eps < \lambda_k \le \lambda_{\ceil{2NW}} \le \tfrac{1}{2}$, i.e., $\eps < \tfrac{1}{2}$. Therefore, $\eps < \lambda_k \le \lambda_{\ceil{2NW}} \le \tfrac{1}{2} < 1-\eps$, i.e., both $k$ and $\ceil{2NW}$ are in the transition region $\{k' : \eps < \lambda_{k'} < 1-\eps\}$, and thus, so are all the indices $k'$ between $k$ and $\ceil{2NW}$. Hence, $$\#\{k' : \eps < \lambda_{k'} < 1-\eps\} \ge \#\{k' : \ceil{2NW} \le k' \le k\} = k-\ceil{2NW}+1.$$ 
However, since $\eps < \tfrac{1}{2}$, by Theorem~\ref{thm:BoundWithoutW} we have $$\#\{k' : \eps < \lambda_{k'} < 1-\eps\} \le 2\ceil{\dfrac{1}{\pi^2}\log(4N)\log\left(\dfrac{4}{\eps(1-\eps)}\right)} < \dfrac{2}{\pi^2}\log(4N)\log\left(\dfrac{8}{\eps}\right)+2 = k-\ceil{2NW}+1.$$ This is a contradiction. Therefore, $$\lambda_k \le \eps = 8\exp\left[-\dfrac{k-\ceil{2NW}-1}{\tfrac{2}{\pi^2}\log(4N)}\right] \quad \text{for} \quad \ceil{2NW} \le k \le N-1.$$

In a similar manner, we can assume $\lambda_k > \eps$, where $$\eps = 10\exp\left[-\dfrac{k-\ceil{2NW}-6}{\tfrac{2}{\pi^2}\log(100NW+25)}\right],$$ and then invoke Theorem~\ref{thm:BoundWithW} to obtain a contradiction. Therefore, $$\lambda_k \le \eps = 10\exp\left[-\dfrac{k-\ceil{2NW}-6}{\tfrac{2}{\pi^2}\log(100NW+25)}\right] \quad \text{for} \quad \ceil{2NW} \le k \le N-1.$$
Combining these two bounds establishes the second part of Corollary~\ref{cor:DPSSEigBounds}.

\subsection{Bounds on $\sum_{k = 0}^{K-1}(1-\lambda_k)$ for $K \le \floor{2NW}$}
For any integer $K$ such that $1 \le K \le \floor{2NW}$, we can apply the first part of the lower bound for $\lambda_k$ in Corollary~\ref{cor:DPSSEigBounds} along with the inequality $\tfrac{e^{-x}}{1-e^{-x}} \le \tfrac{1}{x}$ for $x > 0$ to obtain
\begin{align*}
\sum_{k = 0}^{K-1}(1-\lambda_k) &\le \sum_{k = 0}^{K-1}8\exp\left[-\dfrac{\floor{2NW}-k-2}{\tfrac{2}{\pi^2}\log(4N)}\right]
\\
&\le \sum_{k = -\infty}^{K-1}8\exp\left[-\dfrac{\floor{2NW}-k-2}{\tfrac{2}{\pi^2}\log(4N)}\right]
\\
&= \dfrac{8\exp\left[-\dfrac{\floor{2NW}-K-1}{\tfrac{2}{\pi^2}\log(4N)}\right]}{1-\exp\left[-\dfrac{1}{\tfrac{2}{\pi^2}\log(4N)}\right]}
\\
&\le \dfrac{16}{\pi^2}\log(4N)\exp\left[-\dfrac{\floor{2NW}-K-2}{\tfrac{2}{\pi^2}\log(4N)}\right].
\end{align*}

In a similar manner, we can apply the second part of the lower bound for $\lambda_k$ in Corollary~\ref{cor:DPSSEigBounds} instead of the first part of the lower bound to obtain
\begin{align*}
\sum_{k = 0}^{K-1}(1-\lambda_k) &\le \sum_{k = 0}^{K-1}10\exp\left[-\dfrac{\floor{2NW}-k-7}{\tfrac{2}{\pi^2}\log(100NW+25)}\right]
\\
&\le \sum_{k = -\infty}^{K-1}8\exp\left[-\dfrac{\floor{2NW}-k-7}{\tfrac{2}{\pi^2}\log(100NW+25)}\right]
\\
&= \dfrac{8\exp\left[-\dfrac{\floor{2NW}-K-6}{\tfrac{2}{\pi^2}\log(100NW+25)}\right]}{1-\exp\left[-\dfrac{1}{\tfrac{2}{\pi^2}\log(100NW+25)}\right]}
\\
&\le \dfrac{20}{\pi^2}\log(100NW+25)\exp\left[-\dfrac{\floor{2NW}-K-7}{\tfrac{2}{\pi^2}\log(100NW+25)}\right].
\end{align*}
Combining these two bounds establishes the first part of Corollary~\ref{cor:DPSSEigSumBounds}.

\subsection{Bounds on $\sum_{k = K}^{N-1}\lambda_k$ for $K \ge \ceil{2NW}$}
For any integer $K$ such that $\ceil{2NW} \le K \le N-1$, we can apply the first part of the upper bound for $\lambda_k$ in Corollary~\ref{cor:DPSSEigBounds} along with the inequality $\tfrac{e^{-x}}{1-e^{-x}} \le \tfrac{1}{x}$ for $x > 0$ to obtain
\begin{align*}
\sum_{k = K}^{N-1}\lambda_k &\le \sum_{k = K}^{N-1}8\exp\left[-\dfrac{k-\ceil{2NW}-1}{\tfrac{2}{\pi^2}\log(4N)}\right]
\\
&\le \sum_{k = K}^{\infty}8\exp\left[-\dfrac{k-\ceil{2NW}-1}{\tfrac{2}{\pi^2}\log(4N)}\right]
\\
&= \dfrac{8\exp\left[-\dfrac{K-\ceil{2NW}-1}{\tfrac{2}{\pi^2}\log(4N)}\right]}{1-\exp\left[-\dfrac{1}{\tfrac{2}{\pi^2}\log(4N)}\right]}
\\
&\le \dfrac{16}{\pi^2}\log(4N)\exp\left[-\dfrac{K-\ceil{2NW}-2}{\tfrac{2}{\pi^2}\log(4N)}\right].
\end{align*}

In a similar manner, we can apply the second part of the upper bound for $\lambda_k$ in Corollary~\ref{cor:DPSSEigBounds} instead of the first part of the upper bound to obtain
\begin{align*}
\sum_{k = K}^{N-1}\lambda_k &\le \sum_{k = K}^{N-1}10\exp\left[-\dfrac{k-\ceil{2NW}-6}{\tfrac{2}{\pi^2}\log(100NW+25)}\right]
\\
&\le \sum_{k = K}^{\infty}10\exp\left[-\dfrac{k-\ceil{2NW}-6}{\tfrac{2}{\pi^2}\log(100NW+25)}\right]
\\
&= \dfrac{10\exp\left[-\dfrac{K-\ceil{2NW}-6}{\tfrac{2}{\pi^2}\log(100NW+25)}\right]}{1-\exp\left[-\dfrac{1}{\tfrac{2}{\pi^2}\log(100NW+25)}\right]}
\\
&\le \dfrac{20}{\pi^2}\log(100NW+25)\exp\left[-\dfrac{K-\ceil{2NW}-7}{\tfrac{2}{\pi^2}\log(100NW+25)}\right].
\end{align*}
Combining these two bounds establishes the second part of Corollary~\ref{cor:DPSSEigSumBounds}.

%% file: PSWFEigenvalueBounds.tex
\section{Proof of PSWF Eigenvalue Bounds}
\label{sec:PSWFProofs}

First, we state a result by Boulsane, Bourguiba, and Karoui \cite{Boulsane20} which quantifies how close the PSWF eigenvalues are to DPSS eigenvalues with the same time-bandwidth product. 

\begin{lemma}
\label{lem:DPSS_PSWF}
For any $N \in \N$ and $W \in (0,\tfrac{1}{2})$, $$\left(\sum_{k = 0}^{\infty}\left|\lambda_k(N,W)-\widetilde{\lambda}_k(\pi NW)\right|^2\right)^{1/2} \le \dfrac{4\pi^2W^3}{3\sin(2\pi W)},$$ where we define $\lambda_k(N,W) = 0$ for $k \ge N$ for ease of notation. 
\end{lemma}

In particular, we note that this implies $$\left|\lambda_k\left(N,\dfrac{c}{\pi N}\right)-\widetilde{\lambda}_k(c)\right| \le \left(\sum_{k = 0}^{\infty}\left|\lambda_k\left(N,\dfrac{c}{\pi N}\right)-\widetilde{\lambda}_k(c)\right|^2\right)^{1/2} \le \delta_{c,N} := \dfrac{4c^3}{3\pi N^3\sin(\tfrac{2c}{N})}$$ for all $c > 0$ and all integers $N > \tfrac{2c}{\pi}$ and $k \ge 0$. Also, for any $c > 0$, we have $\delta_{c,N} \searrow 0$ as $N \to \infty$, and thus, $$\displaystyle\lim_{N \to \infty}\lambda_k(N,\dfrac{c}{\pi N}) = \widetilde{\lambda}_k(c).$$ 

\subsection{Proof of Theorem~\ref{thm:PSWFEigGap}}
Since $\left|\lambda_k\left(N,\tfrac{c}{\pi N}\right)-\widetilde{\lambda}_k(c)\right| \le \delta_{c,N}$, we have $$\eps < \widetilde{\lambda}_k(c) < 1-\eps \implies \eps - \delta_{c,N} < \lambda_k\left(N,\dfrac{c}{\pi N}\right) < 1-\eps+\delta_{c,N}$$ for all integers $k \ge 0$. For sufficiently large $N$, $\delta_{c,N} < \eps$ and so, we may apply Theorem~\ref{thm:BoundWithW} for $W = \tfrac{c}{\pi N}$ to obtain 
\begin{align*}
\#\left\{k : \eps < \widetilde{\lambda}_k(c) < 1-\eps\right\} &\le \#\left\{k : \eps-\delta_{c,N} < \lambda_k\left(N,\dfrac{c}{\pi N}\right) < 1-\eps+\delta_{c,N}\right\} 
\\
&\le \dfrac{2}{\pi^2}\log\left(\dfrac{100c}{\pi}+25\right)\log\left(\dfrac{5}{(\eps-\delta_{c,N})(1-\eps+\delta_{c,N})}\right)+7.
\end{align*}
Since this bound holds for all sufficiently large $N$, we may take the limit as $N \to \infty$ to obtain 
\begin{align*}
\#\left\{k : \eps < \widetilde{\lambda}_k(c) < 1-\eps\right\} &\le \lim_{N \to \infty}\left[\dfrac{2}{\pi^2}\log\left(\dfrac{100c}{\pi}+25\right)\log\left(\dfrac{5}{(\eps-\delta_{c,N})(1-\eps+\delta_{c,N})}\right)+7\right]
\\
&\le \dfrac{2}{\pi^2}\log\left(\dfrac{100c}{\pi}+25\right)\log\left(\dfrac{5}{\eps(1-\eps)}\right)+7.
\end{align*}

\subsection{Proof of PSWF eigenvalue bounds (Corollary~\ref{cor:PSWFEigBounds})}
For any $0 \le k \le \floor{\tfrac{2c}{\pi}}-1$, we can apply Corollary~\ref{cor:DPSSEigBounds} for any $N > \tfrac{2c}{\pi}$ and $W = \tfrac{c}{\pi N}$ to obtain $$\lambda_k\left(N,\dfrac{c}{\pi N}\right) \ge 1-10\exp\left[-\dfrac{\floor{\tfrac{2c}{\pi}}-k-7}{\tfrac{2}{\pi^2}\log\left(\tfrac{100c}{\pi}+25\right)}\right].$$ Since this holds for all $N > \tfrac{2c}{\pi}$, we have $$\widetilde{\lambda}_k(c) = \lim_{N \to \infty}\lambda_k\left(N,\dfrac{c}{\pi N}\right) \ge 1-10\exp\left[-\dfrac{\floor{\tfrac{2c}{\pi}}-k-7}{\tfrac{2}{\pi^2}\log\left(\tfrac{100c}{\pi}+25\right)}\right].$$

Similarly, for any $k \ge \ceil{\tfrac{2c}{\pi}}$, we can apply Corollary~\ref{cor:DPSSEigBounds} for any $N > k$ and $W = \tfrac{c}{\pi N}$ to obtain
and $$\lambda_k\left(N,\dfrac{c}{\pi N}\right) \le 10\exp\left[-\dfrac{k-\ceil{\tfrac{2c}{\pi}}-6}{\tfrac{2}{\pi^2}\log\left(\tfrac{100c}{\pi}+25\right)}\right].$$ Since this holds for all $N > k$, we have $$\widetilde{\lambda}_k(c) = \lim_{N \to \infty}\lambda_k\left(N,\dfrac{c}{\pi N}\right) \le 10\exp\left[-\dfrac{k-\ceil{\tfrac{2c}{\pi}}-6}{\tfrac{2}{\pi^2}\log\left(\tfrac{100c}{\pi}+25\right)}\right].$$

\subsection{Proof of PSWF eigenvalue sum bounds (Corollary~\ref{cor:PSWFEigSumBounds})}
For any integer $K$ such that $1 \le K \le \floor{\tfrac{2c}{\pi}}$, we can apply the lower bound for $\widetilde{\lambda}_k(c)$ in Corollary~\ref{cor:PSWFEigBounds} along with the inequality $\tfrac{e^{-x}}{1-e^{-x}} \le \tfrac{1}{x}$ for $x > 0$ to obtain
\begin{align*}
\sum_{k = 0}^{K-1}\left(1-\widetilde{\lambda}_k(c)\right) &\le \sum_{k = 0}^{K-1}10\exp\left[-\dfrac{\floor{\tfrac{2c}{\pi}}-k-7}{\tfrac{2}{\pi^2}\log\left(\tfrac{100c}{\pi}+25\right)}\right]
\\
&\le \sum_{k = -\infty}^{K-1}10\exp\left[-\dfrac{\floor{\tfrac{2c}{\pi}}-k-7}{\tfrac{2}{\pi^2}\log\left(\tfrac{100c}{\pi}+25\right)}\right]
\\
&= \dfrac{10\exp\left[-\dfrac{\floor{\tfrac{2c}{\pi}}-k-6}{\tfrac{2}{\pi^2}\log\left(\tfrac{100c}{\pi}+25\right)}\right]}{1-\exp\left[-\dfrac{1}{\tfrac{2}{\pi^2}\log\left(\tfrac{100c}{\pi}+25\right)}\right]}
\\
&\le \dfrac{20}{\pi^2}\log\left(\dfrac{100c}{\pi}+25\right)\exp\left[-\dfrac{\floor{\tfrac{2c}{\pi}}-K-7}{\tfrac{2}{\pi^2}\log\left(\tfrac{100c}{\pi}+25\right)}\right].
\end{align*}

Similarly, for any integer $K \ge \ceil{\tfrac{2c}{\pi}}$, we can apply the upper bound for $\widetilde{\lambda}_k(c)$ in Corollary~\ref{cor:PSWFEigBounds} along with the inequality $\tfrac{e^{-x}}{1-e^{-x}} \le \tfrac{1}{x}$ for $x > 0$ to obtain
\begin{align*}
\sum_{k = K}^{\infty}\widetilde{\lambda}_k(c) &\le \sum_{k = K}^{\infty}10\exp\left[-\dfrac{k-\ceil{\tfrac{2c}{\pi}}-6}{\tfrac{2}{\pi^2}\log\left(\tfrac{100c}{\pi}+25\right)}\right]
\\
&= \dfrac{10\exp\left[-\dfrac{k-\ceil{\tfrac{2c}{\pi}}-6}{\tfrac{2}{\pi^2}\log\left(\tfrac{100c}{\pi}+25\right)}\right]}{1-\exp\left[-\dfrac{1}{\tfrac{2}{\pi^2}\log\left(\tfrac{100c}{\pi}+25\right)}\right]}
\\
&\le \dfrac{20}{\pi^2}\log\left(\dfrac{100c}{\pi}+25\right)\exp\left[-\dfrac{K-\ceil{\tfrac{2c}{\pi}}-7}{\tfrac{2}{\pi^2}\log\left(\tfrac{100c}{\pi}+25\right)}\right].
\end{align*}

%% file: NumericalResults.tex
\section{Numerical Results}
\label{sec:NumericalResults}
We demonstrate the quality of our bounds on the width of the transition region $\#\{k : \eps < \lambda_k < 1-\eps\}$ with some numerical computations. First, we  fix $W = \tfrac{1}{4}$ (a large value of $W$), and for each integer $2^4 \le N \le 2^{16}$ we use the method described in \cite{Gruenbacher94} to compute $\lambda_k$ for a range $k_{\text{min}} \le k \le k_{\text{max}}$ such that $\lambda_{k_{\text{min}}} > 1-10^{-13}$ and $\lambda_{k_{\text{max}}} < 10^{-13}$. From this, we can determine $\#\{k : \eps < \lambda_k < 1-\eps\}$ for $\eps = 10^{-3}, 10^{-8}, 10^{-13}$. We plot $\#\{k : \eps < \lambda_k < 1-\eps\}$ as well as the upper bound on $\#\{k : \eps < \lambda_k < 1-\eps\}$ from Theorem~\ref{thm:BoundWithoutW} in Figure~\ref{fig:GapsizeVsN}. We note that over this range of parameters, the difference between the bound in Theorem~\ref{thm:BoundWithoutW} and the true width of the transition region $\#\{k : \eps < \lambda_k < 1-\eps\}$ is between $1$ and $14$.

\begin{center}
\begin{figure}[ht]
\includegraphics[scale=0.5]{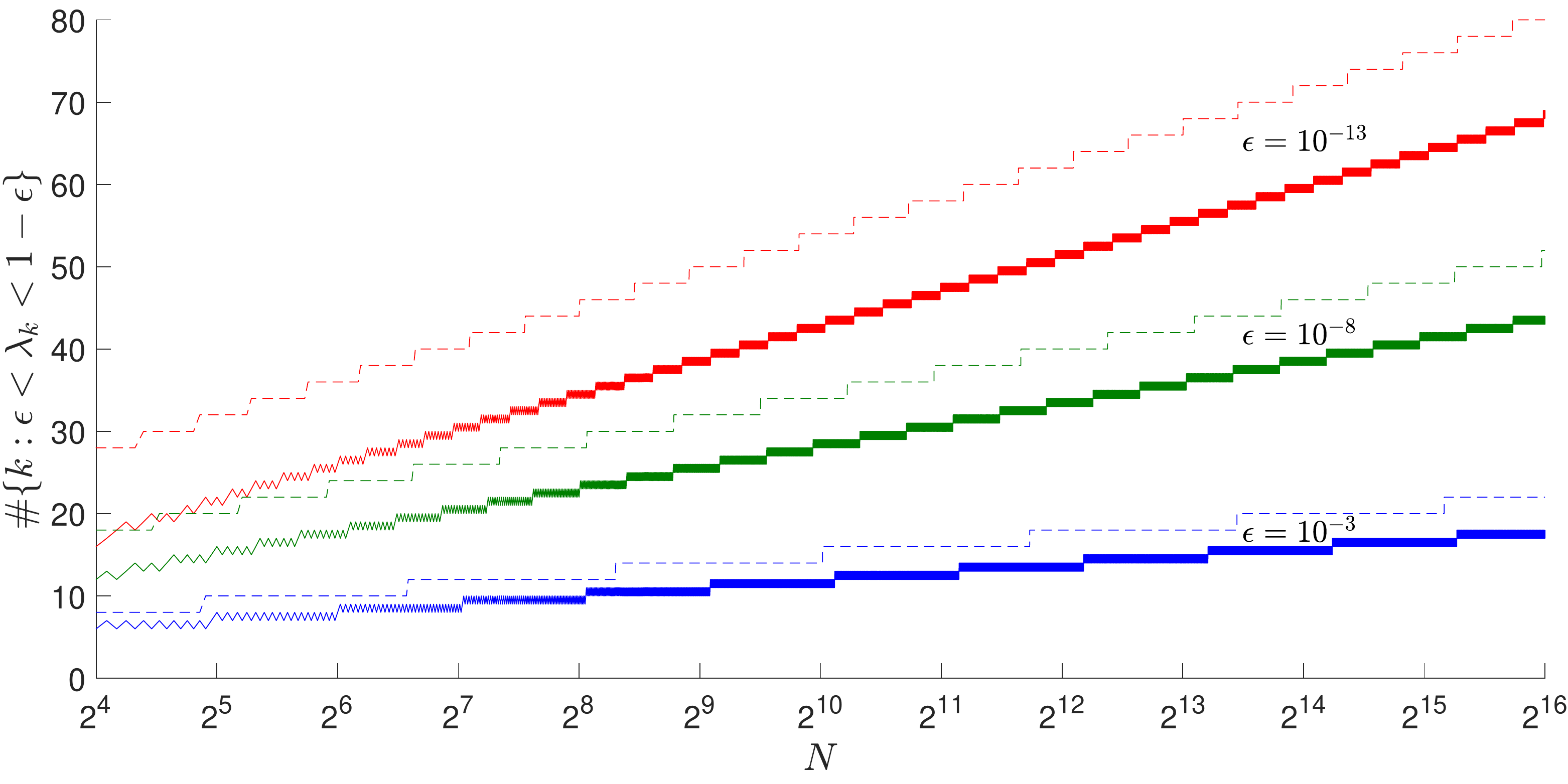}
\caption{Plots of the width of the transition region $\#\{k : \eps < \lambda_k < 1-\eps\}$ vs. $N$ where $W = \tfrac{1}{4}$ and $\eps = 10^{-3}$(blue), $\eps = 10^{-8}$(green), and $10^{-13}$(red) are fixed. The dashed lines indicate the upper bound from Theorem~\ref{thm:BoundWithoutW}.}
\label{fig:GapsizeVsN}
\end{figure}
\end{center}

Next, we  fix $N = 2^{16}$ and for $10001$ logarithmically spaced values of $W$ between $2^{-14}$ and $2^{-2}$, we use the method described in \cite{Gruenbacher94} to compute $\lambda_k$ for a range $k_{\text{min}} \le k \le k_{\text{max}}$ such that $\lambda_{k_{\text{min}}} > 1-10^{-13}$ and $\lambda_{k_{\text{max}}} < 10^{-13}$. From this, we can determine $\#\{k : \eps < \lambda_k < 1-\eps\}$ for $\eps = 10^{-3}, 10^{-8}, 10^{-13}$. We plot $\#\{k : \eps < \lambda_k < 1-\eps\}$ as well as the upper bound on $\#\{k : \eps < \lambda_k < 1-\eps\}$ from Theorem~\ref{thm:BoundWithW} in Figure~\ref{fig:GapsizeVsW}. We note that over this range of parameters, the difference between the bound in Theorem~\ref{thm:BoundWithoutW} and the true width of the transition region $\#\{k : \eps < \lambda_k < 1-\eps\}$ is between $\approx 9.8$ and $\approx 30.7$. 

\begin{center}
\begin{figure}
\includegraphics[scale=0.5]{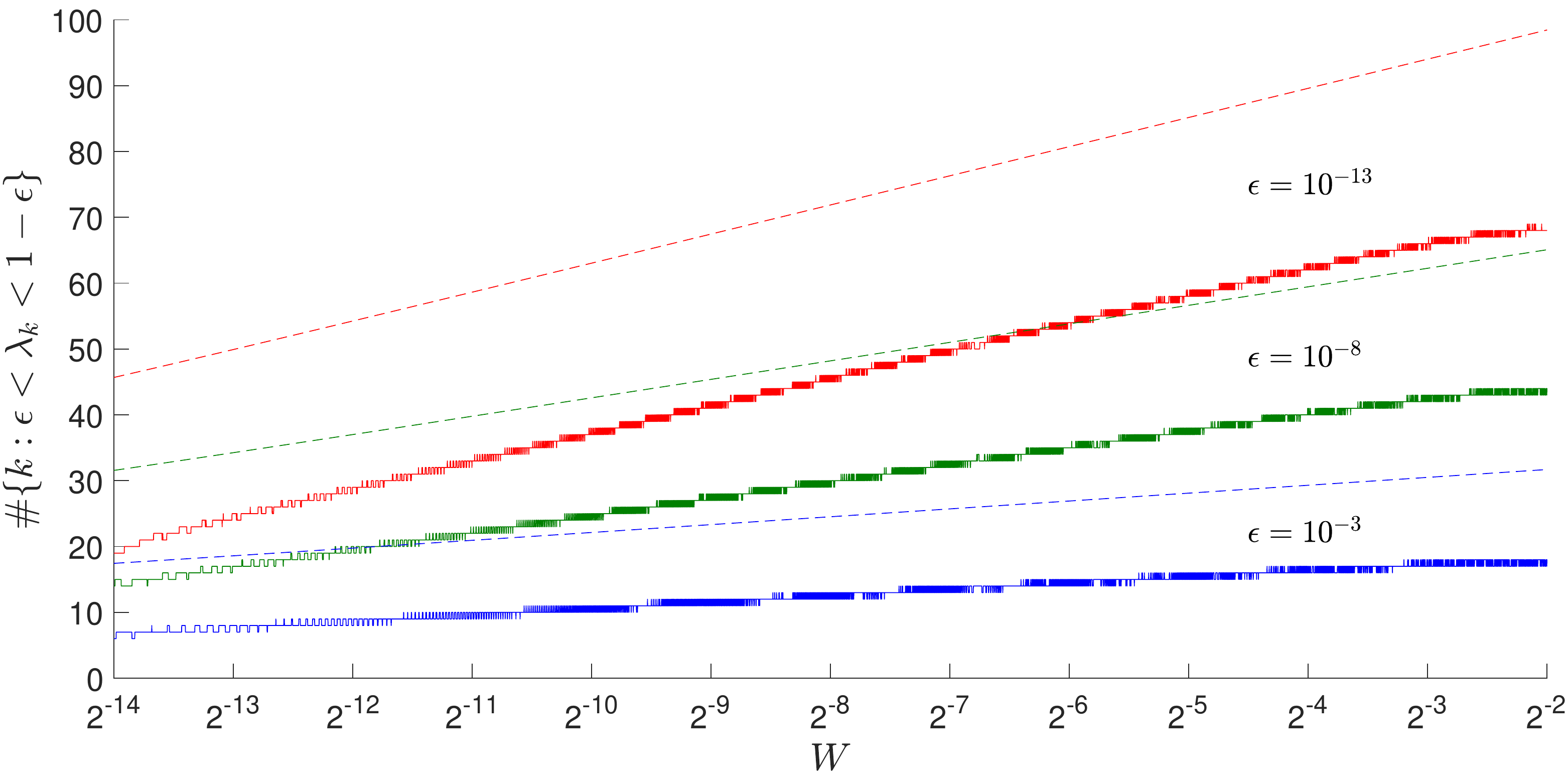}
\caption{Plots of the width of the transition region $\#\{k : \eps < \lambda_k < 1-\eps\}$ vs. $W$ where $N = 2^{16}$ and $\eps = 10^{-3}$(blue), $\eps = 10^{-8}$(green), and $10^{-13}$(red) are fixed. The dashed lines indicate the upper bound from \cref{thm:BoundWithW}.}
\label{fig:GapsizeVsW}
\end{figure}
\end{center}

In Figure~\ref{fig:GapsizeVsN}, we see that the plots of both $2\ceil{\tfrac{1}{\pi^2}\log(4N)\log(\tfrac{4}{\eps(1-\eps)})}$ (the bound in Theorem~\ref{thm:BoundWithoutW}) and the actual width of the transition region increase roughly linearly with $\log N$ and at roughly the same rate. However, the difference between the bound in Theorem~\ref{thm:BoundWithoutW} and the actual width of the transition region is noticeably larger for smaller values of $\eps$ than for larger values of $\eps$. This provides numerical evidence that for a large bandwidth $W$, the bound's dependence on $N$ is close to correct, but the dependence on $\eps$ has some room for improvement. 

In Figure~\ref{fig:GapsizeVsW}, we see that the plots of both $\tfrac{2}{\pi^2}\log(100NW+25)\log\left(\tfrac{5}{\eps(1-\eps)}\right)+7$ and the actual width of the transition region increase roughly linearly with $\log(NW)$ and at roughly the same rate. However, the difference between the bound in Theorem~\ref{thm:BoundWithW} and the actual width of the transition region is quite noticeable. This provides numerical evidence that the leading constant of $\tfrac{2}{\pi^2}$ is indeed correct, but that the other constants leave significant room for improvement. 

Finally, we note that for the range of parameters in both plots, the non-asymptotic bounds on the width of the transition region given by (\ref{eq:ZhuWakin}), (\ref{eq:Boulsane}), and (\ref{eq:FST}) (in Section~\ref{sec:PrevDPSSBounds}) would all be well above the range of the plots in Figures~\ref{fig:GapsizeVsN} and \ref{fig:GapsizeVsW}. The bounds in (\ref{eq:ZhuWakin}) and (\ref{eq:Boulsane}) are proportional to $\tfrac{1}{\eps(1-\eps)}$. Thus, they are only useful when $\eps$ is not too small. Also, the bound in (\ref{eq:FST}) is rather large since the leading constant $\tfrac{8}{\pi^2}$ being $4$ times larger than that in Theorems~\ref{thm:BoundWithoutW} and \ref{thm:BoundWithW}, and the trailing constant $12$ dominates $\tfrac{8}{\pi^2}\log(8N)$ when $N$ isn't too large. In particular, for $\eps = 10^{-3}$ and any $N \in \N$, if we impose the mild constraint that $NW \ge \tfrac{1}{2}$, then the bound in (\ref{eq:ZhuWakin}) is at least $\tfrac{4/\pi^2}{\eps(1-\eps)} \approx 405$, the bound in (\ref{eq:Boulsane}) is at least $\tfrac{0.45-1/6}{\eps(1-\eps)} \approx 283$, and the bound in (\ref{eq:FST}) is at least $(\tfrac{8}{\pi^2}\log(8)+12)\log(\tfrac{15}{\eps}) \approx 131$.

%% file: Zolotarev.tex
\section{Zolotarev numbers}
\label{sec:Zolotarev}
In this section, we review some properties of Zolotarev numbers, which will be useful in our analysis in Appendix~\ref{sec:LowRankDisplacement}. With the exception of Corollary~\ref{cor:UnboundedIntervals}, all the results here have been proven elsewhere. However, we state these results and outline the proofs for sake of completeness. 

For any integer $k \ge 0$, we let $\mathcal{R}_{k,k}$ denote the set of rational functions $\varphi(z) = \tfrac{p(z)}{q(z)}$ such that $p(z)$ and $q(z)$ are polynomials with degree at most $k$. For any two disjoint, closed subsets of the Riemann sphere $E,F \subset \C \cup \{\infty\}$, the Zolotarev number $Z_k(E,F)$ is defined as $$Z_k(E,F) = \inf_{\varphi \in \mathcal{R}_{k,k}}\dfrac{\sup\limits_{z \in E}|\varphi(z)|}{\inf\limits_{z \in F}|\varphi(z)|}.$$ Note that any rational function $\varphi(z) = \tfrac{p(z)}{q(z)}$ can be extended to a continuous function on the Riemann sphere $\C \cup \{\infty\}$ by defining $\varphi(\infty) = \lim\limits_{|z| \to \infty}\varphi(z)$ and $\varphi(z) = \infty$ for any $z$ such that $q(z) = 0$.  

Beckermann and Townsend \cite{BT19} proved the following bound on the Zolotarev numbers for the intervals $E = [-b,-a]$ and $F = [a,b]$.
\begin{theorem}{\cite{BT19}}
\label{thm:ZolotarevBound}
For any reals $b > a > 0$, and any integer $k \ge 0$, $$Z_k([-b,-a],[a,b]) \le 4\exp\left[-\dfrac{\pi^2 k}{\log(\tfrac{4b}{a})} \right].$$
\end{theorem}
The proof of Theorem~\ref{thm:ZolotarevBound} involves using theory of elliptic functions to construct a rational function $\varphi \in \mathcal{R}_{k,k}$ for which $\dfrac{\sup\limits_{z \in [-b,-a]}|\varphi(z)|}{\inf\limits_{z \in [a,b]}|\varphi(z)|} \le 4\exp\left[-\dfrac{\pi^2 k}{\log(\tfrac{4b}{a})} \right].$

A fact about Zolotarev numbers is that they are invariant under invertible M\"{o}bius transforms \cite{Akhiezer90}. 
\begin{lemma} 
\label{lem:Mobius}
For any two disjoint, closed subsets of the Riemann sphere $E,F \subset \C \cup \{\infty\}$ and any M\"{o}bius transform $\phi(z) = \dfrac{\beta_1 z + \beta_2}{\beta_3 z + \beta_4}$ such that $\beta_1\beta_4 \neq \beta_2\beta_3$, we have $Z_k(\phi(E),\phi(F)) = Z_k(E,F)$ for all integers $k \ge 0$.
\end{lemma}
This fact is easily proved by noting that $\varphi^* \in \mathcal{R}_{k,k}$ is the extremal rational function for $(\phi(E),\phi(F))$ if and only if $\varphi^* \circ \phi \in \mathcal{R}_{k,k}$ is the extremal rational function for $(E,F)$.

Using this fact, Beckermann and Townsend proved the following bound on the Zolotarev numbers for two non-overlapping intervals.
\begin{cor}{\cite{BT19}}
\label{cor:BoundedIntervals}
For any two intervals $[c_1,c_2]$ and $[d_1,d_2]$ that are nonoverlapping, and any integer $k \ge 0$, $$Z_k([c_1,c_2],[d_1,d_2]) \le 4\exp\left[-\dfrac{\pi^2 k}{\log(16\gamma)} \right] \quad \text{where} \quad \gamma = \dfrac{(d_1-c_1) (d_2-c_2)}{(d_2-c_1)(d_1-c_2)}.$$
\end{cor}
\begin{proof}
It is trivial to check that $\gamma > 1$ when $[c_1,c_2]$ and $[d_1,d_2]$ do not overlap. Now, set $\alpha = 2\gamma - 1 + 2\sqrt{\gamma^2-\gamma}$ and define the M\"{o}bius transforms $$\phi_1(z) = \dfrac{(d_2-d_1)(z-c_2)}{(d_2-c_2)(z-d_1)} \quad \text{and} \quad \phi_2(z) = \dfrac{(\alpha-1)(z+1)}{(\alpha+1)(z-1)}.$$ One can check that $\phi_1([c_1,c_2]) = [0,\tfrac{\gamma-1}{\gamma}] = [0,\tfrac{(\alpha-1)^2}{(\alpha+1)^2}] = \phi_2([-\alpha,-1])$ and $\phi_1([d_1,d_2]) = [1,\infty] = \phi_2([1,\alpha])$, and that both $\phi_1$ and $\phi_2$ are bijections. Thus, the M\"{o}bius transform $\phi = \phi_2^{-1} \circ \phi_1$ satisfies $\phi([c_1,c_2]) = [-\alpha,-1]$ and $\phi([d_1,d_2]) = [1,\alpha]$. So by applying Theorem~\ref{thm:ZolotarevBound}, Lemma~\ref{lem:Mobius}, and the bound $\alpha = 2\gamma - 1 + 2\sqrt{\gamma^2-\gamma} \le 4\gamma$, we have $$Z_k([c_1,c_2],[d_1,d_2]) = Z_k([-\alpha,-1],[1,\alpha]) \le 4\exp\left[-\dfrac{\pi^2 k}{\log(4\alpha)} \right] \le 4\exp\left[-\dfrac{\pi^2 k}{\log(16\gamma)} \right].$$
\end{proof}

In a nearly identical manner, we can also prove the following bound.
\begin{cor}
\label{cor:UnboundedIntervals}
For any real numbers $c_1 < d_1 < d_2 < c_2$, and any integer $k \ge 0$, $$Z_k([-\infty,c_1]\cup[c_2,\infty],[d_1,d_2]) \le 4\exp\left[-\dfrac{\pi^2 k}{\log(16\gamma)} \right] \quad \text{where} \quad \gamma = \dfrac{(c_2-d_1)(d_2-c_1)}{(c_2-d_2)(d_1-c_1)}.$$
\end{cor}
\begin{proof}
Again, since $c_1 < d_1 < d_2 < c_2$, we have $\gamma > 1$. Now, set $\alpha = 2\gamma - 1 + 2\sqrt{\gamma^2-\gamma}$ and define the M\"{o}bius transforms $$\phi_1(z) = \dfrac{(d_2-d_1)(z-c_1)}{(d_2-c_1)(z-d_1)} \quad \text{and} \quad \phi_2(z) = \dfrac{(\alpha-1)(z+1)}{(\alpha+1)(z-1)}.$$ One can check that $\phi_1([-\infty,c_1]\cup[c_2,\infty]) = [0,\tfrac{\gamma-1}{\gamma}] = [0,\tfrac{(\alpha-1)^2}{(\alpha+1)^2}] = \phi_2([-\alpha,-1])$ and $\phi_1([d_1,d_2]) = [1,\infty] = \phi_2([1,\alpha])$, and that both $\phi_1$ and $\phi_2$ are bijections. Thus, the M\"{o}bius transform $\phi = \phi_2^{-1} \circ \phi_1$ satisfies $\phi([-\infty,c_1]\cup[c_2,\infty]) = [-\alpha,-1]$ and $\phi([d_1,d_2]) = [1,\alpha]$. So by applying Theorem~\ref{thm:ZolotarevBound}, Lemma~\ref{lem:Mobius}, and the bound $\alpha = 2\gamma - 1 + 2\sqrt{\gamma^2-\gamma} \le 4\gamma$, we have $$Z_k([-\infty,c_1]\cup[c_2,\infty],[d_1,d_2]) = Z_k([-\alpha,-1],[1,\alpha]) \le 4\exp\left[-\dfrac{\pi^2 k}{\log(4\alpha)} \right] \le 4\exp\left[-\dfrac{\pi^2 k}{\log(16\gamma)} \right].$$
\end{proof}

%% file: LowRankDisplacement.tex
\section{Singular values of matrices with low rank displacement}
\label{sec:LowRankDisplacement}
With the exception of Theorem~\ref{thm:UnboundedSpec}, the results in this section have all been proven elsewhere. Furthermore, the proof of Theorem~\ref{thm:UnboundedSpec} is very similar to that of Theorem~\ref{thm:BoundedSpec}. However, we state these results and the proof of Theorem~\ref{thm:UnboundedSpec} for sake of completeness.

Throughout this section, we suppose that $\mX \in \C^{M \times N}$ satisfies the displacement equation $$\mC\mX-\mX\mD = \mU\mV^*,$$ where $\mC \in \C^{M \times M}$ and $\mD \in \C^{N \times N}$ are normal matrices, and $\mU \in \C^{M \times \nu}$ and $\mV \in \C^{N \times \nu}$ (where it is understood that $\nu \ll \min\{M,N\}$ for the results in this section to be useful). Our goal is to show that $\mX$ is approximately low-rank under certain assumptions on $\Spec(\mC)$ and $\Spec(\mD)$. 

Beckermann and Townsend \cite{BT19} showed that the numerical rank of $\mX$ can be bounded in terms of Zolotarev numbers.
\begin{theorem}{\cite{BT19}}
\label{thm:SingularValueDecay}
If $\Spec(\mC) \subset E$ and $\Spec(\mD) \subset F$, then the singular values of $\mX$ satisfy $$\sigma_{\nu k+j}(\mX) \le \sigma_j(\mX)Z_k(E,F)$$ for any integers $j \ge 1$, $k \ge 0$.
\end{theorem}
The proof involves showing that for any rational function $\varphi \in \mathcal{R}_{k,k}$, we can construct a rank-$(\nu k+j-1)$ matrix $\mY$ such that $$\mX-\mY = \varphi(\mC)(\mX-\mX_{j-1})\varphi(\mD)^{-1}$$ where $\mX_{j-1}$ is the best rank-$(j-1)$ approximation to $\mX$. Then, by applying the facts that $$\|\varphi(\mC)\| \le \sup\limits_{z \in E}|\varphi(z)|, \quad \|\varphi(\mD)^{-1}\| \le \sup\limits_{z \in F}|\varphi(z)^{-1}| = \left(\inf\limits_{z \in F}|\varphi(z)|\right)^{-1}, \quad \text{and} \quad \|\mX-\mX_{j-1}\| = \sigma_j(\mX)$$ along with the submultiplicativity of the matrix norm, we obtain $$\sigma_{\nu k+j}(\mX) \le \|\mX-\mY\| \le \|\varphi(\mC)\| \cdot \|\mX-\mX_{j-1}\| \cdot \|\varphi(\mD)^{-1}\| \le \sigma_j(\mX) \cdot \dfrac{\sup\limits_{z \in E}|\varphi(z)|}{\inf\limits_{z \in F}|\varphi(z)|}.$$ This bound holds for any $\varphi \in \mathcal{R}_{k,k}$. Taking the infimum over all $\varphi \in \mathcal{R}_{k,k}$ yields $\sigma_{\nu k+j}(\mX) \le \sigma_j(\mX)Z_k(E,F)$.

By combining Theorem~\ref{thm:SingularValueDecay} (with $j = 1$) along with Corollary~\ref{cor:BoundedIntervals}, Beckermann and Townsend established the following result.
\begin{theorem}{\cite{BT19}}
\label{thm:BoundedSpec}
If $\text{Spec}(\mC) \subset [c_1,c_2]$ and $\text{Spec}(\mD) \subset [d_1,d_2]$ where $[c_1,c_2]$ and $[d_1,d_2]$ are nonoverlapping, then for any integer $k \ge 0$, $$\sigma_{\nu k+1}(\mX) \le 4\|\mX\|\exp\left[-\dfrac{\pi^2 k}{\log(16\gamma)}\right] \quad \text{where} \quad \gamma = \dfrac{(d_1-c_1) (d_2-c_2)}{(d_2-c_1)(d_1-c_2)}.$$
\end{theorem}

Finally, by combining Theorem~\ref{thm:SingularValueDecay} (with $j = 1$) along with Corollary~\ref{cor:UnboundedIntervals}, we obtain Theorem~\ref{thm:UnboundedSpec} (stated in Section~\ref{sec:ProofThm1}).

%% file: SincApprox.tex
\section{Polynomial approximations of the sinc function}
\label{sec:SincFunction}
For a bandwidth parameter $W > 0$, we define the sinc function $$g(t) = \dfrac{\sin(2\pi Wt)}{\pi t} \quad \text{for} \quad t \in \R.$$
First, we prove the following bound on the derivatives of $g(t)$.
\begin{lemma}
\label{lem:SincDerivatives}
For any non-negative integer $k$, $$\left|g^{(k)}(t)\right| \le (2\pi W)^k\min\left\{\dfrac{2W}{k+1} , \dfrac{2}{\pi|t|} \right\} \quad \text{for all} \quad t \in \R.$$ 
\end{lemma}
\begin{proof}
For $k = 0$, we can apply the inequality $|\sin \theta| \le \min\{|\theta|,2\}$ to obtain $|g(t)| \le \min\{2W,\tfrac{2}{\pi |t|}\}$. Hence, we can proceed with the case where $k \ge 1$. Note that we can write the sinc function as $$g(t) = \dfrac{\sin(2\pi Wt)}{\pi t} = \int_{-W}^{W}e^{j2\pi ft}\,df.$$ By differentiating under the integral sign $k$ times, we obtain $$g^{(k)}(t) = \dfrac{d^k}{dt^k}\left[\int_{-W}^{W}e^{j2\pi ft}\,df\right] = \int_{-W}^{W}\dfrac{d^k}{dt^k}\left[e^{j2\pi ft}\right]\,df = \int_{-W}^{W}(j2\pi f)^ke^{j2\pi ft}\,df.$$ Applying the triangle inequality yields the bound $$\left|g^{(k)}(t)\right| = \left|\int_{-W}^{W}(j2\pi f)^ke^{j2\pi ft}\,df\right| \le \int_{-W}^{W}\left|(j2\pi f)^ke^{j2\pi ft}\right|\,df = \int_{-W}^{W}(2\pi|f|)^k\,df = \dfrac{(2\pi W)^{k+1}}{\pi(k+1)}.$$
Alternatively, we can use integration by parts before applying the triangle inequality to obtain \begin{align*}\left|g^{(k)}(t)\right| &= \left|\int_{-W}^{W}(j2\pi f)^ke^{j2\pi ft}\,df\right| 
\\
&= \left|\left[(j2\pi f)^k\dfrac{e^{j2\pi ft}}{j 2\pi t}\right]_{f = -W}^{f = W} - \int_{-W}^{W}j2\pi k(j2\pi f)^{k-1}\dfrac{e^{j2\pi ft}}{j 2\pi t}\,df\right|
\\
&= \left|\dfrac{(j2\pi W)^ke^{j2\pi Wt}}{j2\pi t}-\dfrac{(-j2\pi W)^ke^{-j2\pi Wt}}{j2\pi t} - \int_{-W}^{W}j2\pi k(j2\pi f)^{k-1}\dfrac{e^{j2\pi ft}}{j 2\pi t}\,df\right|
\\
&\le \left|\dfrac{(j2\pi W)^ke^{j2\pi Wt}}{j2\pi t}\right| + \left|\dfrac{(-j2\pi W)^ke^{-j2\pi Wt}}{j2\pi t}\right| + \left|\int_{-W}^{W}j2\pi k(j2\pi f)^{k-1}\dfrac{e^{j2\pi ft}}{j 2\pi t}\,df\right|
\\
&\le \left|\dfrac{(j2\pi W)^ke^{j2\pi Wt}}{j2\pi t}\right| + \left|\dfrac{(-j2\pi W)^ke^{-j2\pi Wt}}{j2\pi t}\right| + \int_{-W}^{W}\left|j2\pi k(j2\pi f)^{k-1}\dfrac{e^{j2\pi ft}}{j 2\pi t}\right|\,df
\\
&= \dfrac{(2\pi W)^k}{2\pi |t|} + \dfrac{(2\pi W)^k}{2\pi |t|} + \int_{-W}^{W}\dfrac{k}{|t|}(2\pi|f|)^{k-1}\,df
\\
&= \dfrac{(2\pi W)^k}{\pi |t|} + \dfrac{(2\pi W)^k}{\pi |t|}
\\
&= \dfrac{2(2\pi W)^k}{\pi |t|}.
\end{align*}
Combining the two bounds yields $$\left|g^{(k)}(t)\right| \le \min\left\{\dfrac{(2\pi W)^{k+1}}{\pi(k+1)} , \dfrac{2(2\pi W)^k}{\pi|t|} \right\} = (2\pi W)^k\min\left\{\dfrac{2W}{k+1} , \dfrac{2}{\pi|t|} \right\}.$$
\end{proof}

We finish this section by noting a well-known theorem on Chebyshev interpolation. 
\begin{theorem}\cite{SuliMayers03}
\label{thm:ChebyshevInterp}
Suppose $g \in C^k[a,b]$ for some positive integer $k$. Define the Chebyshev interpolating polynomial of degree $k-1$ by $$P_k(t) = \sum_{m = 1}^{k}g(t_m)\prod_{\substack{m' = 1,\ldots,k \\ m' \neq m}}\dfrac{t-t_{m'}}{t_m-t_{m'}}$$ where $$t_m = \dfrac{b+a}{2} + \dfrac{b-a}{2}\cos\left(\dfrac{2m-1}{2k}\pi\right) \quad \text{for} \quad m = 1,\ldots,k$$ are the Chebyshev nodes used for interpolation. Then, for any $t \in [a,b]$, we have $$\left|g(t)-P_k(t)\right| \le \dfrac{(b-a)^k}{2^{2k-1}k!}\max_{\xi \in [a,b]}\left|g^{(k)}(\xi)\right|.$$ 
\end{theorem}
We will use both Lemma~\ref{lem:SincDerivatives} and Theorem~\ref{thm:ChebyshevInterp} in Section~\ref{sec:ProofThm2} to prove Theorem~\ref{thm:BoundWithW}.